\documentclass{amsart}
\usepackage{amscd,thmdefs,amsmath,amssymb,latexsym, amsfonts,txfonts}
\usepackage{textcomp}
\usepackage{color}
\usepackage{hyperref}
\usepackage{tikz,pgfplots,pgf} 
\tikzset{node distance=2cm, auto}
\usetikzlibrary{calc} 
\usetikzlibrary{trees} 

\newtheorem{cs}{Case}

\def\N{\mathbb{N}}
\def\R{\mathbb{R}}

\def\K{\mathbb{K}}

\def\F{\mathcal{F}}

\def\L{\mathcal{L}}

\def\Re{\mathrm{Re}}

\def\lin{\mathrm{lin}}
\def\Lip{\mathrm{Lip}}
\def\Lipo{\mathrm{Lip}_0}
\def\LipoF{\mathrm{Lip}_{0F}}

\def\rank{\mathrm{rank}}

\newcommand{\n}[1]{ \left\|#1\right\| }
\newcommand{\bign}[1]{ \big\|#1\big\| }
\newcommand{\pair}[2]{{\langle #1, #2 \rangle}}
\newcommand{\Bpair}[2]{{\Big\langle #1, #2 \Big\rangle}}

\begin{document}
\title[Lipschitz tensor product]{Lipschitz tensor product}

\author{M. G. Cabrera-Padilla}
\address{Departamento de Matem\'{a}ticas, Universidad de Almer\'{i}a, 04120 Almer\'{i}a, Spain}
\email{m\_gador@hotmail.com}

\author{J. A. Ch\'avez-Dom\'{\i}nguez}
\address{Department of Mathematics,
University of Texas at Austin,
Austin, TX 78712-1202.}
\email{jachavezd@math.utexas.edu}
\thanks{The second author's research was partially supported by NSF grants DMS-1001321 and DMS-1400588, and the third author's one by Junta of Andaluc\'{i}a grant FQM-194 and Ministerio de Econom\'{i}a y Competitividad project MTM2014-58984-P.}

\author{A. Jim\'{e}nez-Vargas}
\address{Departamento de Matem\'{a}ticas, Universidad de Almer\'{i}a, 04120 Almer\'{i}a, Spain}
\email{ajimenez@ual.es}
\date{\today}

\author{Mois\'{e}s  Villegas-Vallecillos}
\address{Departamento de Matem\'{a}ticas, Universidad de C\'{a}diz, 11510 Puerto Real, Spain}
\email{moises.villegas@uca.es}



\begin{abstract}
Inspired by ideas of R. Schatten in his celebrated monograph \cite{schatten} on a theory of cross-spaces, we introduce the notion of a Lipschitz tensor product $X\boxtimes E$ of a pointed metric space $X$ and a Banach space $E$ as a certain linear subspace of the algebraic dual of $\Lipo(X,E^*)$. 
We prove that $X\boxtimes E$ is linearly isomorphic to the linear space of all finite-rank continuous linear operators from $(X^\#,\tau_p)$ into $E$, where $X^\#$ denotes the space $\Lipo(X,\mathbb{K})$ and $\tau_p$ is the topology of pointwise convergence of $X^\#$. The concept of Lipschitz tensor product of elements of $X^\#$ and $E^*$ yields the space $X^\#\boxast E^*$ as a certain linear subspace of the algebraic dual of $X\boxtimes E$. To ensure the good behavior of a norm on $X\boxtimes E$ with respect to the Lipschitz tensor product of Lipschitz functionals (mappings) and bounded linear functionals (operators), the concept of dualizable (respectively, uniform) Lipschitz cross-norm on $X\boxtimes E$ is defined. We show that the Lipschitz injective norm $\varepsilon$, the Lipschitz projective norm $\pi$ and the Lipschitz $p$-nuclear norm $d_p$ $(1\leq p\leq\infty)$ are uniform dualizable Lipschitz cross-norms on $X\boxtimes E$. In fact, $\varepsilon$ is the least dualizable Lipschitz cross-norm and $\
pi$ is the greatest Lipschitz cross-norm on $X\boxtimes E$. Moreover, dualizable Lipschitz cross-norms $\alpha$ on $X\boxtimes E$ are characterized by satisfying the relation $\varepsilon\leq\alpha\leq\pi$.
In addition, the Lipschitz injective (projective) norm on $X\boxtimes E$ can be identified with the injective (respectively, projective) tensor norm on the Banach-space tensor product between the Lipschitz-free space over $X$ and $E$.
In terms of the space $X^\#\boxast E^*$, we describe the spaces of Lipschitz compact (finite-rank, approximable) operators from $X$ to $E^*$.
\end{abstract}
\maketitle

\section*{Introduction}

The Lipschitz space $\Lipo(X,E)$ is the Banach space of all Lipschitz maps $f$ from a pointed metric space $X$ to a Banach space $E$ that vanish at the base point of $X$, under the Lipschitz norm given by 
$$
\Lip(f)=\sup\left\{\frac{\left\|f(x)-f(y)\right\|}{d(x,y)}\colon x,y\in X, \; x\neq y\right\}.
$$ 
The elements of $\Lipo(X,E)$ are referred to as Lipschitz operators. If $\K$ is the field of real or complex numbers, the space $\Lipo(X,\mathbb{K})$, denoted by $X^\#$, is called the Lipschitz dual of $X$. A comprehensive reference for the basic theory of the spaces of Lipschitz functions is the book \cite{w} by N. Weaver.

The use of techniques of the theory of algebraic tensor product of Banach spaces to tackle the problem of the duality for Lipschitz operators from $X$ to $E$ goes back to the seventies with the works \cite{j70,j71} of J. A. Johnson. Recently, the second-named author \cite{ch11} has adopted this approach to describe the duals of spaces of Lipschitz $p$-summing operators from $X$ to $E^*$ for $1\leq p\leq\infty$. The notion of Lipschitz $p$-summing operators between metric spaces, a nonlinear generalization of $p$-summing operators, was introduced by J. D. Farmer and W. B. Johnson in \cite{fj}, where a nonlinear version of the Pietsch factorization theorem was established. The article \cite{fj} has motivated the study of Lipschitz versions of different types of bounded linear operators in the works \cite{ch11,ch12,cz11,cz12,jsv}.

The reading of the paper \cite{ch11} invites to give a definition for the tensor product of $X$ and $E$. In \cite{j70}, J. A. Johnson proved that the dual of the closed linear subspace of $\Lipo(X,E^*)^*$ spanned by the functionals $\delta_x\boxtimes e$ on $\Lipo(X,E^*)$ with $x\in X$ and $e\in E$, defined by
$(\delta_x\boxtimes e)(f)=\langle f(x),e\rangle$, is isometrically isomorphic to $\Lipo(X,E^*)$. It is well known that the dual of the projective tensor product of Banach spaces $E$ and $F$ can be identified with the space of all bounded linear operators from $E$ to $F^*$, so the predual of $\Lipo(X,E^*)$ provided by Johnson's result plays the role of the projective tensor product in the linear theory and this fact suggests to call Lipschitz tensor product of $X$ and $E$ the linear subspace of the algebraic dual of $\Lipo(X,E^*)$ spanned by the functionals $\delta_x\boxtimes e$. 
In \cite{ch11}, this space is called the space of $E$-valued molecules on $X$, a generalization of the Arens--Eells space $\mbox{\AE}(X)$ of scalar-valued molecules on $X$ (see \cite{ae,w}). 

Our purpose is to develop a theory of the Lipschitz tensor product $X\boxtimes E$ of a pointed metric space $X$ and a Banach space $E$ by following the original ideas of R. Schatten \cite{schatten} used to construct the algebraic tensor product of two Banach spaces. We are also motivated by the problem of researching the spaces of Lipschitz compact (finite-rank, approximable) operators from $X$ to $E^*$ introduced in \cite{jsv}.

We now describe the contents of this paper. In Section \ref{1}, we introduce and study the Lipschitz tensor product $X\boxtimes E$. We show that $\left\langle X\boxtimes E,\Lipo(X,E^*)\right\rangle$ forms a dual pair and identify linearly the space $\Lipo(X,E^*)$ with a linear subspace of the algebraic dual of $X\boxtimes E$, and the space $X\boxtimes E$ with the space $\mathcal{F}((X^\#,\tau_p);E)$ of all finite-rank continuous linear operators from $(X^\#,\tau_p)$ to $E$, where $\tau_p$ denotes the topology of pointwise convergence of $X^\#$. 

In Section \ref{2}, we define the concept of a Lipschitz tensor product of a Lipschitz functional $g\in X^\#$ and a bounded linear functional $\phi\in E^*$ as the linear functional $g\boxtimes\phi$ on $X\boxtimes E$ given by 
$$
(g\boxtimes\phi)\left(\sum_{i=1}^n \delta_{(x_i,y_i)}\boxtimes e_i\right)=\sum_{i=1}^n(g(x_i)-g(y_i))\left\langle \phi,e_i\right\rangle,
$$
and consider the associated Lipschitz tensor product of $X\boxtimes E$, denoted by $X^\#\boxast E^*$, as the linear subspace of the algebraic dual of $X\boxtimes E$ spanned by the elements $g\boxtimes\phi$. It is showed that the space $X^\#\boxast E^*$ is linearly isomorphic to the space $\LipoF(X,E^*)$ of all Lipschitz finite-rank operators from $X$ to $E^*$. Moreover, we give the notion of a Lipschitz tensor product of a base-point preserving Lipschitz map between pointed metric spaces $h\colon X\to Y$ and a bounded linear operator between Banach spaces $T\colon E\to F$ as the linear operator $h\boxtimes T$ from $X\boxtimes E$ to $Y\boxtimes F$, defined by 
$$
(h\boxtimes T)\left(\sum_{i=1}^n\delta_{(x_i,y_i)}\boxtimes e_i\right)=\sum_{i=1}^n\delta_{(h(x_i),h(y_i))}\boxtimes T(e_i).
$$

In Section \ref{3}, we consider a norm $\alpha$ on $X\boxtimes E$ and obtain a normed space $X\boxtimes_\alpha E$ and its completion $X\widehat{\boxtimes}_\alpha E$. We are interested in the called Lipschitz cross-norms which are those satisfying the condition: 
$$
\alpha\left(\delta_{(x,y)}\boxtimes e\right)=d(x,y)\left\|e\right\|\qquad \left(x,y\in X,\; e\in E\right),
$$
where $\delta_{(x,y)}\boxtimes e$ is the linear functional on $\Lipo(X,E^*)$ of the form
$$
\left(\delta_{(x,y)}\boxtimes e\right)(f)=\left\langle f(x)-f(y),e\right\rangle .
$$
Desirable attributes for Lipschitz cross-norms on $X\boxtimes E$ is that they behave well with respect to the formation of Lipschitz tensor products of functionals and operators. In this line, we introduce dualizable Lipschitz cross-norms and uniform Lipschitz cross-norms. A Lipschitz cross-norm $\alpha$ on $X\boxtimes E$ is called dualizable if given $g\in X^\#$ and $\phi\in E^*$, we have
$$
\left|\sum_{i=1}^n\left(g(x_i)-g(y_i)\right)\left\langle \phi,e_i\right\rangle\right|
\leq\Lip(g)\left\|\phi\right\|\alpha\left(\sum_{i=1}^n \delta_{(x_i,y_i)}\boxtimes e_i\right)
$$
for all $\sum_{i=1}^n \delta_{(x_i,y_i)}\boxtimes e_i\in X\boxtimes E$, and it is called uniform if given $h\in\Lipo(X,X)$ and $T\in\L(E;E)$, we have
$$
\alpha\left(\sum_{i=1}^n\delta_{(h(x_i),h(y_i))}\boxtimes T(e_i)\right)
\leq\Lip(h)\left\|T\right\|\alpha\left(\sum_{i=1}^n \delta_{(x_i,y_i)}\boxtimes e_i\right)
$$ 
for all $\sum_{i=1}^n \delta_{(x_i,y_i)}\boxtimes e_i\in X\boxtimes E$. Given a dualizable Lipschitz cross-norm $\alpha$ on $X\boxtimes E$, we also construct a norm $\alpha'$ on $ X^\#\boxast E^*$ called the associated Lipschitz norm of $\alpha$, since it satisfies the condition: 
$$
\alpha'(g\boxtimes\phi)=\Lip(g)\left\|\phi\right\|\qquad \left(g\in X^\#,\;\phi\in E^*\right).
$$
The space $X^\#\boxast E^*$ with the norm $\alpha'$ will be denoted by $X^\#\boxast_{\alpha'} E^*$ and its completion by $X^\#\widehat{\boxast}_{\alpha'} E^*$.

In Sections \ref{4} and \ref{5}, we investigate, respectively, the dual norm $L$ induced on $X\boxtimes E$ by the norm $\Lip$ of $\Lipo(X,E^*)$:
$$
L\left(\sum_{i=1}^n \delta_{(x_i,y_i)}\boxtimes e_i\right)
=\sup\left\{\left|\sum_{i=1}^n \left\langle f(x_i)-f(y_i),e_i\right\rangle\right|\colon f\in\Lipo(X,E^*),\ \Lip(f)\leq 1 \right\},
$$
and the Lipschitz injective norm $\varepsilon$ on $X\boxtimes E$:
$$
\varepsilon\left(\sum_{i=1}^n \delta_{(x_i,y_i)}\boxtimes e_i\right)=\sup\left\{\left|\sum_{i=1}^n\left(g(x_i)-g(y_i)\right)\left\langle\phi,e_i\right\rangle\right|\colon g\in B_{X^\#}, \; \phi\in B_{E^*}\right\}.
$$
The Lipschitz projective norm $\pi$ and the Lipschitz $p$-nuclear norm $d_p$ for $1\leq p\leq\infty$ are defined on $u\in X\boxtimes E$ as 
\begin{align*}
\pi(u)&=\inf\left\{\sum_{i=1}^n d(x_i,y_i)\left\|e_i\right\|\colon u=\sum_{i=1}^n\delta_{(x_i,y_i)}\boxtimes e_i\right\},\\
d_1(u)&=\inf\left\{\left(\sup_{g\in B_{X^\#}}\left(\max_{1\leq i\leq n}\lambda_i\left|g(x_i)-g(y_i)\right|\right)\right)\left(\sum_{i=1}^n\left\|e_i\right\|\right)\colon u=\sum_{i=1}^n \lambda_i\delta_{(x_i,y_i)}\boxtimes e_i,\; \left\{\lambda_i\right\}_{i=1}^n\subset\R^+\right\},\\
d_p(u)&=\inf\left\{\left(\sup_{g\in B_{X^\#}}\left(\sum_{i=1}^n\lambda_i^p\left|g(x_i)-g(y_i)\right|^p\right)^{\frac{1}{p'}}\right)\left(\sum_{i=1}^n\left\|e_i\right\|^p\right)^{\frac{1}{p}}\colon u=\sum_{i=1}^n \lambda_i\delta_{(x_i,y_i)}\boxtimes e_i,\; \left\{\lambda_i\right\}_{i=1}^n\subset\R^+\right\}\quad (1<p<\infty),\\
d_\infty(u)&=\inf\left\{\left(\sup_{g\in B_{X^\#}}\left(\sum_{i=1}^n\lambda_i\left|g(x_i)-g(y_i)\right|\right)\right)\left(\max_{1\leq i \leq n}\left\|e_i\right\|\right)\colon u=\sum_{i=1}^n \lambda_i\delta_{(x_i,y_i)}\boxtimes e_i,\; \left\{\lambda_i\right\}_{i=1}^n\subset\R^+\right\},
\end{align*}
where the infimum is taken over all such representations of $u$. Sections \ref{6} and \ref{7} are devoted to their study. All those norms on $X\boxtimes E$ are uniform and dualizable Lipschitz cross-norms. In fact, $\varepsilon$ is the least dualizable Lipschitz cross-norm and $\pi$ is the greatest Lipschitz cross-norm on $X\boxtimes E$. Furthermore, dualizable Lipschitz cross-norms $\alpha$ on $X\boxtimes E$ are characterized as those which satisfy the relation $\varepsilon\leq\alpha\leq\pi$. We also prove that $L$ agrees with $\pi$ and justify the terminologies ``injective" and ``projective" for the Lipschitz norms $\varepsilon$ and $\pi$, respectively.  
We identify $X\boxtimes_{\varepsilon} E$ and its completion $X\widehat{\boxtimes}_{\varepsilon} E$ with $\mathcal{F}((X^\#,\tau_p);E)$ and its closure in the operator norm topology, respectively.
We also show that the Lipschitz injective (projective) norm on $X\boxtimes E$ can be identified with the injective (respectively, projective) tensor norm on the Banach-space tensor product between the Lipschitz-free space over $X$ and $E$.

In Section \ref{8}, we deal with the space $\LipoF(X,E^*)$ of all Lipschitz finite-rank operators from $X$ to $E^*$ and its closure in the Lipschitz norm topology, the space of Lipschitz approximable operators from $X$ to $E^*$. It is proved that the former space is isometrically isomorphic to $X^\#\boxast_{\pi'} E^*$, and the latter to $X^\#\widehat{\boxast}_{\pi'} E^*$. The approximation property for Banach spaces was introduced by Grothendieck in his famous memory \cite{g}. We show that if $X^\#$ has the approximation property, then the space of all Lipschitz compact operators from $X$ to $E^*$ is isometrically isomorphic to $X^\#\widehat{\boxast}_{\pi'}E^*$ for any Banach space $E$. We close the paper giving a new expression of the norm $\pi'$.

\smallskip

\textbf{Notation.} Given two metric spaces $(X,d_X)$ and $(Y,d_Y)$, let us recall that a map $f\colon X\to Y$ is said to be Lipschitz if there is a real constant $C\geq 0$ such that $d_Y(f(x),f(y))\leq Cd_X(x,y)$ for all $x,y\in X$. The least constant $C$ for which the preceding inequality holds will be denoted by $\Lip(f)$, that is, 
$$
\Lip(f)=\sup\left\{\frac{d_Y(f(x),f(y))}{d_X(x,y)}\colon x,y\in X, \; x\neq y\right\}.
$$ 
A pointed metric space $X$ is a metric space with a base point in $X$ which we always will denote by $0$. We will consider a Banach space $E$ over $\K$ as a pointed metric space with the zero vector as the base point. As is customary, $B_E$ and $S_E$ stand for the closed unit ball of $E$ and the unit sphere of $E$, respectively. Given two pointed metric spaces $X$ and $Y$,  $\Lipo(X,Y)$ denotes the set of all base-point preserving Lipschitz maps from $X$ to $Y$. 

For two linear spaces $E$ and $F$, $L(E;F)$ stands for the linear space of all linear operators from $E$ into $F$. In the case that $E$ and $F$ are locally convex Hausdorff spaces, we denote by $\L(E;F)$ the vector space of all continuous linear operators from $E$ into $F$, and by $\mathcal{F}(E;F)$ its subspace of finite-rank operators. When $F=\K$, we write $E'$ instead of $L(E;\K)$ and $E^*$ in place of $\L(E;\K)$. Unless stated otherwise, if $E$ and $F$ are Banach spaces, $\L(E;F)$ is endowed with its natural norm topology. For each $e\in E$ and $\phi\in E'$, we frequently will write $\langle\phi,e\rangle$ instead of $\phi(e)$.


\section{Lipschitz tensor products}\label{1}

The Lipschitz tensor product of a pointed metric space $X$ and a Banach space $E$, which we will denote from now on by $X\boxtimes E$, can be constructed as a space of linear functionals on $\Lipo(X,E^*)$. 

\begin{definition}\label{def-Lipschitz-tensor-product}
Let $X$ be a pointed metric space and $E$ a Banach space. For each $x\in X$, let $\delta_{(x,0)}\colon\Lipo(X,E^*)\to E^*$ be the linear map defined by
$$
\delta_{(x,0)}(f)=f(x) \qquad \left(f\in\Lipo(X,E^*)\right).
$$
For each $(x,y)\in X^2$, let $\delta_{(x,y)}\colon\Lipo(X,E^*)\to E^*$ be the linear map given by
$$
\delta_{(x,y)}=\delta_{(x,0)}-\delta_{(y,0)}.
$$
Let $\Delta(X,E^*)$ denote the linear subspace of $L(\Lipo(X,E^*);E^*)$ spanned by the set $\left\{\delta_{(x,y)}\colon (x,y)\in X^2\right\}$. For any $\gamma\in\Delta(X,E^*)$ and $e\in E$, let $\gamma\boxtimes e\colon\Lipo(X,E^*)\to\K$ be the linear functional given by 
$$
(\gamma\boxtimes e)(f)=\left\langle\gamma(f),e\right\rangle\qquad\left(f\in\Lipo(X,E^*)\right).
$$ 
In particular, for any $(x,y)\in X^2$ and $e\in E$, let $\delta_{(x,y)}\boxtimes e$ be the element of $\Lipo(X,E^*)'$ defined by
$$
\left(\delta_{(x,y)}\boxtimes e\right)(f)
=\left\langle \delta_{(x,y)}(f),e\right\rangle
=\left\langle \delta_{(x,0)}(f)-\delta_{(y,0)}(f),e\right\rangle
=\left\langle f(x)-f(y),e\right\rangle , \qquad \forall f\in\Lipo(X,E^*).
$$
The Lipschitz tensor product $X \boxtimes E$ is defined as the vector subspace of $\Lipo(X,E^*)'$ spanned by the set $$\left\{\delta_{(x,y)}\boxtimes e\colon (x,y)\in X^2,\, e\in E\right\}.$$
\end{definition}

The following properties of the Lipschitz tensor product can be checked easily. 

\begin{lemma}\label{lem-1}
Let $\lambda\in\K$, $(x,y),(x_1,y_1),(x_2,y_2)\in X^2$ and $e,e_1,e_2\in E$. 
\begin{enumerate}
\item $\lambda\left(\delta_{(x,y)}\boxtimes e\right)=(\lambda\delta_{(x,y)})\boxtimes e=\delta_{(x,y)}\boxtimes(\lambda e)$.
\item $\left(\delta_{(x_1,y_1)}+\delta_{(x_2,y_2)}\right)\boxtimes e=\delta_{(x_1,y_1)}\boxtimes e+\delta_{(x_2,y_2)}\boxtimes e$.
\item $\delta_{(x,y)}\boxtimes (e_1+e_2)=\delta_{(x,y)}\boxtimes e_1+\delta_{(x,y)}\boxtimes e_2$.
\item $\delta_{(x,x)}\boxtimes e=\delta_{(x,y)}\boxtimes 0=0$.
\end{enumerate}
\end{lemma}

We say that $\delta_{(x,y)}\boxtimes e$ is an elementary Lipschitz tensor. Note that each element $u$ in $X\boxtimes E$ is of the form $u=\sum_{i=1}^n\lambda_i(\delta_{(x_i,y_i)}\boxtimes e_i)$, where $n\in\N$, $\lambda_i\in\K$, $(x_i,y_i)\in X^2$ and $e_i\in E$. This representation of $u$ is not unique. It is worth noting that each element $u$ of $X\boxtimes E$ can be represented as $u=\sum_{i=1}^n \delta_{(x_i,y_i)}\boxtimes e_i$ since $\lambda(\delta_{(x,y)}\boxtimes e)=\delta_{(x,y)}\boxtimes(\lambda e)$. This representation of $u$ admits the following refinement. 

\begin{lemma}\label{lem-0}
Every nonzero Lipschitz tensor $u\in X\boxtimes E$ has a representation in the form $\sum_{i=1}^m\delta_{(z_i,0)}\boxtimes d_i$, where 
$$
m=\min\left\{k\in\N\colon \exists z_1,\ldots,z_k\in X,\ d_1,\ldots,d_k\in E \; | \; u=\sum_{i=1}^k\delta_{(z_i,0)}\boxtimes d_i\right\}
$$ 
and the points $z_1,\ldots,z_m$ in $X$ are distinct from the base point $0$ of $X$ and pairwise distinct. 
\end{lemma}

\begin{proof}
Let $u\in X\boxtimes E$ and let $\sum_{i=1}^n \delta_{(x_i,y_i)}\boxtimes e_i$ be a representation of $u$. 
We have 
\begin{align*}
u&=\sum_{i=1}^n\delta_{(x_i,y_i)}\boxtimes e_i\\
&=\sum_{i=1}^n \left(\delta_{(x_i,0)}-\delta_{(y_i,0)}\right)\boxtimes e_i\\
&=\sum_{i=1}^n \delta_{(x_i,0)}\boxtimes e_i+\sum_{i=1}^n \delta_{(y_i,0)}\boxtimes (-e_i)\\
&=\sum_{i=1}^n \delta_{(x_i,0)}\boxtimes e_i+\sum_{i=n+1}^{2n} \delta_{(y_{i-n},0)}\boxtimes (-e_{i-n})\\
&=\sum_{i=1}^{2n} \delta_{(z_i,0)}\boxtimes d_i,
\end{align*}
where 
$$
\delta_{(z_i,0)}\boxtimes d_i
=\left\{\begin{array}{lll}
\delta_{(x_i,0)}\boxtimes e_i&\text{ if }& i=1,\ldots,n,\\
\delta_{(y_{i-n},0)}\boxtimes (-e_{i-n})&\text{ if }& i=n+1,\ldots,2n.
\end{array}\right.
$$
Then, by the well-ordering principle of $\N$, there exists a smallest natural number $m$ for which there is a representation of $u$ in the form $\sum_{i=1}^m\delta_{(z_i,0)}\boxtimes d_i$. Since $u\neq 0$, it is clear that $z_i\neq 0$ for some $i\in\{1,\ldots,m\}$. This implies that $z_j\neq 0$ for all $j\in\{1,\ldots,m\}$. Otherwise, observe that $\sum_{i=1,i\neq j}^m\delta_{(z_i,0)}\boxtimes d_i$ would be a representation of $u$ containing $m-1$ terms and this contradicts the definition of $m$. Moreover, if $z_j=z_k$ for some $j,k\in\{1,\ldots,m\}$ with $j\neq k$, we would have
$$
u=\sum_{i=1,j\neq i\neq k}^m\delta_{(z_i,0)}\boxtimes d_i + \left(\delta_{(z_j,0)}\boxtimes (d_j+d_k)\right),
$$
and this is impossible. Hence the points $z_i$ are pairwise distinct.
\end{proof}

We can concatenate the representations of two elements of $X\boxtimes E$ to get a representation of their sum. 

\begin{lemma}\label{lem-conca}
Let $u,v\in X\boxtimes E$ and let $\sum_{i=1}^n \delta_{(x_i,y_i)}\boxtimes e_i$ and $\sum_{i=1}^m \delta_{(x'_i,y'_i)}\boxtimes e'_i$ be representations of $u$ and $v$, respectively. Then $\sum_{i=1}^{n+m} \delta_{(x''_i,y''_i)}\boxtimes e''_i$, where 
$$
(x''_i,y''_i)
=\left\{
\begin{array}{lll}
(x_i,y_i)&\text{ if }& i=1,\ldots,n,\\
(x'_{i-n},y'_{i-n})&\text{ if }& i=n+1,\ldots,n+m,
\end{array}\right.
\qquad 
e''_i
=\left\{
\begin{array}{lll}
e_i&\text{ if }& i=1,\ldots,n,\\
e'_{i-n}&\text{ if }& i=n+1,\ldots,n+m,
\end{array}\right.
$$
is a representation of $u+v$. 
\end{lemma}


We now describe the action of a Lipschitz tensor $u\in X\boxtimes E$ on a function $f\in\Lipo(X,E^*)$. 

\begin{lemma}\label{lem-2}
Let $u=\sum_{i=1}^n \delta_{(x_i,y_i)}\boxtimes e_i\in X\boxtimes E$ and $f\in\Lipo(X,E^*)$. Then  
$$
u(f)=\sum_{i=1}^n\left\langle f(x_i)-f(y_i),e_i\right\rangle .
$$
\end{lemma}


Our next aim is to characterize the zero Lipschitz tensor. For it we need the following Lipschitz operators. 

\begin{lemma}\label{lem-3}
Let $g\in X^\#$ and $\phi\in E^*$. The map $g\cdot\phi\colon X\to E^*$, given by $(g\cdot\phi)(x)=g(x)\phi$, belongs to $\Lipo(X,E^*)$ and $\Lip(g\cdot\phi)=\Lip(g)\left\|\phi\right\|$. 
\end{lemma}

\begin{proof}
Clearly, $g\cdot\phi$ is well defined. Let $x,y\in X$. For any $e\in E$, we obtain 
$$
\left|\left\langle(g\cdot\phi)(x)-(g\cdot\phi)(y),e\right\rangle\right|
=\left|\left\langle\left(g(x)-g(y)\right)\phi,e\right\rangle\right|
=\left|g(x)-g(y)\right|\left|\left\langle\phi,e\right\rangle\right|
\leq\Lip(g)d(x,y)\left\|\phi\right\|\left\|e\right\|,
$$
and so $\left\|(g\cdot\phi)(x)-(g\cdot\phi)(y)\right\|\leq\Lip(g)\left\|\phi\right\|d(x,y)$. Then $g\cdot\phi\in\Lipo(X,E^*)$ and $\Lip(g\cdot\phi)\leq\Lip(g)\left\|\phi\right\|$. For the converse inequality, note that 
$$
\left|g(x)-g(y)\right|\left\|\phi\right\|
=\left\|(g\cdot\phi)(x)-(g\cdot\phi)(y)\right\|
\leq\Lip(g\cdot\phi)d(x,y)
$$ 
for all $x,y\in X$, and therefore $\Lip(g)\left\|\phi\right\|\leq\Lip(g\cdot\phi)$.
\end{proof}

\begin{proposition}\label{pro-0}
If $u=\sum_{i=1}^n \delta_{(x_i,y_i)}\boxtimes e_i\in X\boxtimes E$, then the following assertions are equivalent:
\begin{enumerate}
	\item $u=0$.
	\item $\sum_{i=1}^n\left(g(x_i)-g(y_i)\right)\left\langle \phi,e_i\right\rangle=0$ for every $g\in B_{X^\#}$ and $\phi\in B_{E^*}$.
	\item $\sum_{i=1}^n\left(g(x_i)-g(y_i)\right)e_i=0$ for every $g\in B_{X^\#}$.
\end{enumerate}
\end{proposition}

\begin{proof}
(i) implies (ii): If $u=0$, then $u(f)=0$ for all $f\in\Lipo(X,E^*)$. Since $u=\sum_{i=1}^n \delta_{(x_i,y_i)}\boxtimes e_i$, it follows that $\sum_{i=1}^n\langle f(x_i)-f(y_i),e_i\rangle=0$ for all $f\in\Lipo(X,E^*)$ by Lemma \ref{lem-2}. For any $g\in B_{X^\#}$ and $\phi\in B_{E^*}$, the function $g\cdot\phi$ is in $\Lipo(X,E^*)$ by Lemma \ref{lem-3}, and therefore we have   
$$
\sum_{i=1}^n\left(g(x_i)-g(y_i)\right)\left\langle\phi,e_i\right\rangle
=\sum_{i=1}^n\left\langle\left(g(x_i)-g(y_i)\right)\phi,e_i\right\rangle
=\sum_{i=1}^n\left\langle(g\cdot\phi)(x_i)-(g\cdot\phi)(y_i),e_i\right\rangle
=0.
$$

(ii) implies (iii): If (ii) holds, then $\langle\phi,\sum_{i=1}^n(g(x_i)-g(y_i))e_i\rangle=0$ for every $g\in B_{X^\#}$ and $\phi\in B_{E^*}$. Since $B_{E^*}$ separates the points of $E$, it follows that $\sum_{i=1}^n(g(x_i)-g(y_i))e_i=0$ for all $g\in B_{X^\#}$. 

(iii) implies (i): By Lemma \ref{lem-0}, we can write $u=\sum_{i=1}^m\delta_{(z_i,0)}\boxtimes d_i$, where the points $z_i$ in $X$ are pairwise distinct and different from the base point $0$. It follows that  
$$
\sum_{i=1}^n \delta_{(x_i,y_i)}\boxtimes e_i+\sum_{i=1}^m\delta_{(z_i,0)}\boxtimes (-d_i)=u-u=0,
$$
and, by using the fact proved above that (i) implies (iii), we have  
$$
\sum_{i=1}^n\left(g(x_i)-g(y_i)\right)e_i+\sum_{i=1}^m\left(g(z_i)-g(0)\right)(-d_i)=0
$$
for all $g\in B_{X^\#}$. If (iii) holds, we get that 
$$
\sum_{i=1}^mg(z_i)d_i=\sum_{i=1}^n\left(g(x_i)-g(y_i)\right)e_i=0
$$
for all $g\in B_{X^\#}$. Set 
$$
r=\min\left(\left\{d(z_i,z_j)\colon i,j\in\{1,\ldots,m\},\ i\neq j\right\}\cup\left\{d(z_i,0)\colon i\in\{1,\ldots,m\}\right\}\right).
$$
Clearly, $r>0$. Given $j\in\{1,\ldots,m\}$, define $g_j\colon X\to\R$ by  
$$
g_j(x)=\max\left\{0,r-d(x,z_j)\right\}.
$$
It is easy to check that $g_j\in B_{X^\#}$, $g_j(z_j)=r$ and $g_j(z_i)=0$ for all $i\in\{1,\ldots,m\}\setminus\{j\}$. Hence $
0=\sum_{i=1}^mg_j(z_i)d_i=rd_j$, therefore $d_1=d_2=\cdots=d_m=0$ and thus $u=0$.
\end{proof}

According to Definition \ref{def-Lipschitz-tensor-product}, $X\boxtimes E$ is a linear subspace of $\Lipo(X,E^*)'$. Furthermore, we have the next fact.

\begin{theorem}\label{theo-dual-pair}
$\left\langle X\boxtimes E,\Lipo(X,E^*)\right\rangle$ forms a dual pair, where the bilinear form $\left\langle \cdot,\cdot\right\rangle$ associated to the dual pair is given, for $u=\sum_{i=1}^n \delta_{(x_i,y_i)}\boxtimes e_i\in X\boxtimes E$ and $f\in\Lipo(X,E^*)$, by 
$$
\left\langle u,f\right\rangle=\sum_{i=1}^n \left\langle f(x_i)-f(y_i),e_i\right\rangle .
$$
\end{theorem}

\begin{proof}
Note that $\langle u,f\rangle=u(f)$ by Lemma \ref{lem-2}. It is plain that $\langle\cdot,\cdot\rangle\colon (X\boxtimes E)\times\Lipo(X,E^*)\to\K$ is a well-defined bilinear map and that $\Lipo(X,E^*)$ separates points of $X\boxtimes E$. 
Moreover, if $f\in\Lipo(X,E^*)$ and $\langle u,f\rangle=0$ for all $u\in X\boxtimes E$, then $\left\langle f(x),e\right\rangle=\left\langle\delta_{(x,0)}\boxtimes e,f\right\rangle=0$ for all $x\in X$ and $e\in E$. This implies that $f=0$ and thus $X\boxtimes E$ separates points of $\Lipo(X,E^*)$. This completes the proof of the theorem.
\end{proof}

Since $\left\langle X\boxtimes E,\Lipo(X,E^*)\right\rangle$ is a dual pair, $\Lipo(X,E^*)$ can be identified with a linear subspace of $(X\boxtimes E)'$ as follows. 

\begin{corollary}\label{linearization}
For every map $f\in\Lipo(X,E^*)$, the functional $\Lambda(f)\colon X\boxtimes E\to\K$, given by 
$$
\Lambda(f)(u)
=\sum_{i=1}^n\left\langle f(x_i)-f(y_i),e_i\right\rangle
$$
for $u=\sum_{i=1}^n \delta_{(x_i,y_i)}\boxtimes e_i\in X\boxtimes E$, is linear. We say that $\Lambda(f)$ is the linear functional on $X\boxtimes E$ associated to $f$. The map $f\mapsto \Lambda(f)$ is a linear monomorphism from $\Lipo(X,E^*)$ into $(X\boxtimes E)'$. 
\end{corollary}

\begin{proof}
Let $f\in\Lipo(X,E^*)$. By Theorem \ref{theo-dual-pair}, note that $\Lambda(f)(u)=\left\langle u,f\right\rangle$ for all $u\in X\boxtimes E$. It is immediate that $\Lambda(f)$ is a well-defined linear functional on $X\boxtimes E$ and that $f\mapsto \Lambda(f)$ from $\Lipo(X,E^*)$ into $(X\boxtimes E)'$ is a well-defined linear map. 
Finally, let $f\in\Lipo(X,E^*)$ and assume that $\Lambda(f)=0$. Then $\left\langle u,f\right\rangle=0$ for all $u\in X\boxtimes E$. Since $X\boxtimes E$ separates points of $\Lipo(X,E^*)$, it follows that $f=0$ and this proves that the map $\Lambda$ is one-to-one. 
\end{proof}

We next show that $X\boxtimes E$ is linearly isomorphic to the linear space $\mathcal{F}((X^\#,\tau_p);E)$ of all finite-rank linear operators from $X^\#$ into $E$ which are continuous from the topology of pointwise convergence $\tau_p$ of $X^\#$ to the norm topology of $E$.

\begin{definition}
Let $X$ be a pointed metric space. For $f\in X^\#$, $x\in X$ and $\varepsilon\in\R^+$, we put
$$
B(f,x,\varepsilon)=\left\{g\in X^\#\colon \left|g(x)-f(x)\right|<\varepsilon\right\}.
$$
Let $\mathcal{S}$ be the family of sets $\left\{B(f,x,\varepsilon)\colon f\in X^\#,\, x\in X,\, \varepsilon\in\R^+\right\}$. Then the topology of pointwise convergence $\tau_p$ on $X^\#$ is the topology generated by $\mathcal{S}$. 
\end{definition}

We can check that $(X^\#,\tau_p)$ is a locally convex space. Next we describe its dual space.

\begin{lemma}\label{carac-dual-pointwise}
Let $X$ be a pointed metric space. Then $(X^\#,\tau_p)^*=\lin(\left\{\delta_x\colon x\in X\right\})(\subset (X^\#)')$, where $\delta_x$ is the functional on $X^\#$ defined by $\delta_x(g)=g(x)$.
\end{lemma}

\begin{proof}
Define the linear functional $T\colon X^\#\to\K$ by $T(g)=\sum_{i=1}^n\lambda_ig(x_i)$ for all $g\in X^\#$, where $n\in\N$, $\lambda_1,\ldots\lambda_n\in\K$ and $x_1,\ldots,x_n\in X$. Put $r=1+\sum_{i=1}^n|\lambda_i|$ and let $\varepsilon>0$ be arbitrary. If $g\in\bigcap_{i=1}^nB(0,x_i,\varepsilon/r)$, then $\left|T(g)\right|\leq\sum_{i=1}^n\left|\lambda_i\right|\left|g(x_i)\right|<\varepsilon$. This proves that $T$ is continuous on $X^\#$ when it is equipped with the topology $\tau_p$. Conversely, we need to show that every element $S$ in $(X^\#,\tau_p)^*$ is of that form. Since $S$ is continuous in the $\tau_p$-topology, there is an open neighborhood $V$ of $0$ such that $|S(g)|<1$ for all $g\in V$. We can suppose that $V=\bigcap_{i=1}^nB(0,x_i,\varepsilon)$ for suitable $n\in\N$, $x_1,\ldots,x_n\in X$ and $\varepsilon>0$. Take $f\in\bigcap_{i=1}^n\ker\delta_{x_i}$. Then $mf\in V$ for each $m\in\N$. By the linearity of $S$, it follows that $|S(f)|<1/m$ for all $m\in\N$ and 
 so $S(f)=0$. This shows that $\bigcap_{i=1}^n\ker\delta_{x_i}\subset\ker S$ and the lemma follows from a known fact of linear algebra.
\end{proof}

\begin{theorem}\label{teo-identification}
The map $J\colon X\boxtimes E\to \mathcal{F}((X^\#,\tau_p);E)$, given by 
$$
J(u)(g)=\sum_{i=1}^n\left(g(x_i)-g(y_i)\right)e_i
$$
for $u=\sum_{i=1}^n \delta_{(x_i,y_i)}\boxtimes e_i\in X\boxtimes E$ and $g\in X^\#$, is a linear isomorphism.
\end{theorem}

\begin{proof}
Let $u=\sum_{i=1}^n \delta_{(x_i,y_i)}\boxtimes e_i\in X\boxtimes E$. It is immediate that $J(u)\colon X^\#\to E$ is well defined and linear. 
Note that $J(u)(X^\#)\subset\lin\{e_1,\ldots,e_n\}$ and hence $J(u)$ has finite-dimensional range. In order to prove that $J(u)$ is continuous from $(X^\#,\tau_p)$ to $E$, it is sufficient to see that $J(u)$ is continuous at $0$. Let $\varepsilon>0$. Denote $r=2(1+\sum_{i=1}^n||e_i||)$ and put $z_i=x_i$ for $i=1,\ldots,n$ and $z_i=y_{i-n}$ for $i=n+1,\ldots,2n$. Take $U=\cap_{i=1}^{2n} B(0,z_i,\varepsilon/r)$. For any $g\in U$, we have 
$$
\left\|J(u)(g)\right\|
\leq \sum_{i=1}^n\left(\left|g(x_i)\right|+\left|g(y_i)\right|\right)\left\|e_i\right\|
\leq \sum_{i=1}^n2\frac{\varepsilon}{r}\left\|e_i\right\|
<\varepsilon ,
$$  
as required. Moreover, by Proposition \ref{pro-0}, the map $u\mapsto J(u)$ from $X\boxtimes E$ to $\F((X^\#,\tau_p);E)$ is well defined. 
We now show that $J$ is linear. If $\lambda\in\K$, then $\lambda u=\sum_{i=1}^n \delta_{(x_i,y_i)}\boxtimes(\lambda e_i)$ and so 
$$
J(\lambda u)(g)
=\sum_{i=1}^n\left(g(x_i)-g(y_i)\right)(\lambda e_i)
=\lambda\sum_{i=1}^n\left(g(x_i)-g(y_i)\right) e_i
=\lambda J(u)(g)
$$
for all $g\in X^\#$. Now, let $v=\sum_{i=1}^m \delta_{(x'_i,y'_i)}\boxtimes e'_i\in X\boxtimes E$ and take the representation $\sum_{i=1}^{n+m} \delta_{(x''_i,y''_i)}\boxtimes e''_i$ of $u+v$ given in Lemma \ref{lem-conca}. Then we have  
\begin{align*}
J(u+v)(g)&=\sum_{i=1}^{n+m}\left(g(x''_i)-g(y''_i)\right)e''_i\\
&=\sum_{i=1}^{n}\left(g(x''_i)-g(y''_i)\right)e''_i+\sum_{i=n+1}^{n+m}\left(g(x''_i)-g(y''_i)\right)e''_i\\
&=\sum_{i=1}^{n}\left(g(x_i)-g(y_i)\right)e_i+\sum_{i=n+1}^{n+m}\left(g(x'_{i-n})-g(y'_{i-n})\right)e'_{i-n}\\
&=\sum_{i=1}^{n}\left(g(x_i)-g(y_i)\right)e_i+\sum_{i=1}^{m}\left(g(x'_i)-g(y'_i)\right)e'_i\\
&=J(u)(g)+J(v)(g)
\end{align*}
for all $g\in X^\#$.
It remains to show that $J$ is bijective. On one hand, assume that $J(u)=0$.
Then $J(u)(g)=\sum_{i=1}^n(g(x_i)-g(y_i))e_i=0$ for all $g\in X^\#$, this implies that $u=0$ by Proposition \ref{pro-0} and so $J$ is one-to-one.
On the other hand, if $T\in\mathcal{F}((X^\#,\tau_p);E)$, take a basis $\{e_1,\ldots,e_n\}$ of $T(X^\#)$. For each $g\in X^\#$, there are unique $\lambda^{(g)}_1,\ldots,\lambda^{(g)}_n\in\K$ such that $T(g)=\sum_{i=1}^n \lambda^{(g)}_ie_i$. For each $i\in\{1,\ldots,n\}$, let $y^i\colon T(X^\#)\to\K$ be given by $y^i(T(g))=\lambda^{(g)}_i$ for all $g\in X^\#$. The uniqueness of the representation of each element of $T(X^\#)$ implies that each $y^i$ is linear. Hence $y^i$ is continuous on $T(X^\#)$ since the linear space $T(X^\#)$ is finite-dimensional. Then each $T_i=y^i\circ T$ belongs to $(X^\#,\tau_p)^*$ and $T(g)=\sum_{i=1}^n T_i(g)e_i$ for all $g\in X^\#$. Since $(X^\#,\tau_p)^*=\lin(\left\{\delta_x\colon x\in X\right\})(\subset (X^\#)')$ by Lemma \ref{carac-dual-pointwise}, 
for each $i\in\{1,\ldots,n\}$ there are $m(i)\in\N$, $\lambda_1^{(i)},\ldots,\lambda_{m(i)}^{(i)}\in\mathbb{K}$ and $x_1^{(i)},\ldots,x_{m(i)}^{(i)}\in X$ such that $T_i=\sum_{j=1}^{m(i)}\lambda_j^{(i)}\delta_{x_j^{(i)}}$. Then, for each $g\in X^\#$, we may write    
$$
T(g)
=\sum_{i=1}^n T_i(g)e_i
=\sum_{j=1}^m g(x_j)u_j
=J\left(\sum_{j=1}^m\delta_{(x_j,0)}\boxtimes u_j\right)(g)
$$ 
for certain $m\in\N$, $x_1,\ldots,x_m\in X$ and $u_1,\ldots,u_m\in E$. This proves that $J$ is onto.
\end{proof}

\begin{corollary}\label{coro-teo-identification}
Let $X$ be a pointed metric space. If $E$ is a finite-dimensional Banach space, then $X\boxtimes E$ is linearly isomorphic to $\L((X^\#,\tau_p);E)$. In particular, $X\boxtimes \K$ is linearly isomorphic to $(X^\#,\tau_p)^*=\lin(\left\{\delta_x\colon x\in X\right\})$.
\end{corollary}


\section{Lipschitz tensor product functionals and operators}\label{2}

We first introduce the concept of Lipschitz tensor product functional of a Lipschitz functional and a bounded linear functional.

\begin{definition}\label{def-funct}
Let $X$ be a pointed metric space and $E$ a Banach space. Let $g\in X^\#$ and $\phi\in E^*$. The map $g\boxtimes\phi\colon X\boxtimes E\to\K$, given by 
$$
(g\boxtimes\phi)(u)
=\sum_{i=1}^n\left(g(x_i)-g(y_i)\right)\left\langle \phi,e_i\right\rangle
$$
for $u=\sum_{i=1}^n \delta_{(x_i,y_i)}\boxtimes e_i\in X\boxtimes E$, is called the Lipschitz tensor product functional of $g$ and $\phi$.
\end{definition}

By Lemma \ref{lem-2}, note that  
$$
(g\boxtimes\phi)(u)
=\sum_{i=1}^n\left\langle(g\cdot\phi)(x_i)-(g\cdot\phi)(y_i),e_i\right\rangle
=u(g\cdot\phi).
$$
The following result which follows easily from this formula gathers some properties of these functionals.

\begin{lemma}\label{lem-4}
Let $g\in X^\#$ and $\phi\in E^*$. The functional $g\boxtimes\phi\colon X\boxtimes E\to\K$ is a well-defined linear map satisfying $\lambda(g\boxtimes\phi)=(\lambda g)\boxtimes\phi=g\boxtimes(\lambda\phi)$ for any $\lambda\in\K$. Moreover, $(g_1+g_2)\boxtimes\phi=g_1\boxtimes\phi+g_2\boxtimes\phi$ for all $g_1,g_2\in X^\#$ and $g\boxtimes(\phi_1+\phi_2)=g\boxtimes\phi_1+g\boxtimes\phi_2$ for all $\phi_1,\phi_2\in E^*$. 
\end{lemma}

\begin{definition}
Let $X$ be a pointed metric space and $E$ a Banach space. The space $ X^\#\boxast E^*$ is defined as the linear subspace of $(X\boxtimes E)'$ spanned by the set $\left\{g\boxtimes\phi\colon g\in X^\#,\ \phi\in E^*\right\}$. This space is called the associated Lipschitz tensor product of $X\boxtimes E$.
\end{definition}

From the aforementioned formula we also derive easily the following fact.

\begin{lemma}\label{lem-4.1}
For any $\sum_{j=1}^m g_j\boxtimes\phi_j\in X^\#\boxast E^*$ and $\sum_{i=1}^n \delta_{(x_i,y_i)}\boxtimes e_i\in X\boxtimes E$, we have 
$$
\left(\sum_{j=1}^m g_j\boxtimes\phi_j\right)\left(\sum_{i=1}^n \delta_{(x_i,y_i)}\boxtimes e_i\right)
=\left(\sum_{i=1}^n \delta_{(x_i,y_i)}\boxtimes e_i\right)\left(\sum_{j=1}^m g_j\cdot\phi_j\right).
$$
\end{lemma}

Each element $u^*$ in $ X^\#\boxast E^*$ has the form $u^*=\sum_{j=1}^m \lambda_j(g_j\boxtimes\phi_j)$, where $m\in\N$, $\lambda_j\in\K$, $g_j\in X^\#$ and $\phi_j\in E^*$, but this representation is not unique. Since $\lambda(g\boxtimes\phi)=(\lambda g)\boxtimes\phi=g\boxtimes(\lambda\phi)$, each element of $X^\#\boxast E^*$ can be expressed as $\sum_{j=1}^m g_j\boxtimes\phi_j$. This representation can be improved as follows.

\begin{lemma}\label{lem-indep}
Every nonzero element $u^*$ in $X^\#\boxast E^*$ has a representation $\sum_{j=1}^m g_j\boxtimes\phi_j$ such that the functions $g_1,\ldots,g_m$ in $X^\#$ are nonzero and the functionals $\phi_1,\ldots,\phi_m$ in $E^*$ are linearly independent.
\end{lemma}

\begin{proof}
Let $u^*\in X^\#\boxast E^*$, $u^*\neq 0$. Since $0\boxtimes\phi=0$, we can take a representation for $u^*$, $\sum_{i=1}^n h_i\boxtimes\phi_i$, where $h_1,\ldots,h_n$ are nonzero. If the vectors $\phi_1,\ldots,\phi_n$ are linearly independent, we have finished. Otherwise, take $F=\lin(\{\phi_1,\ldots,\phi_n\})$ and choose a subset of $\{\phi_1,\ldots,\phi_n\}$, which is a basis for $F$, $\phi_1,\ldots,\phi_p$ (after reordering) for some $p<n$. For each $i\in\{p+1,\ldots,n\}$ we can express the vector $\phi_i$ as a unique linear combination in the form $\phi_i=\sum_{k=1}^p\lambda_k^{(i)}\phi_k$, where $\lambda_1^{(i)},\ldots,\lambda_p^{(i)}\in\K$. Using Lemma \ref{lem-4}, we can write 
\begin{align*}
u^*&=\sum_{i=1}^p h_i\boxtimes\phi_i+\sum_{i=p+1}^n h_i\boxtimes\phi_i\\
&=\sum_{i=1}^p h_i\boxtimes\phi_i+\sum_{i=p+1}^n h_i\boxtimes \left(\sum_{k=1}^p\lambda_k^{(i)}\phi_k\right)\\
&=\sum_{i=1}^p h_i\boxtimes\phi_i+\sum_{i=p+1}^n \left(\sum_{k=1}^p\lambda_k^{(i)}\left(h_i\boxtimes \phi_k\right)\right)\\
&=\sum_{i=1}^p h_i\boxtimes\phi_i+ \sum_{k=1}^p\left(\sum_{i=p+1}^n\lambda_k^{(i)}\left(h_i\boxtimes \phi_k\right)\right)\\
&=\sum_{i=1}^p h_i\boxtimes\phi_i+\sum_{k=1}^p\left(\sum_{i=p+1}^n\lambda_k^{(i)}h_i\right)\boxtimes \phi_k\\
&=\sum_{j=1}^p \left(h_j+\sum_{i=p+1}^n\lambda_j^{(i)}h_i\right)\boxtimes \phi_j.
\end{align*}
Denote $g_j=h_j+\sum_{i=p+1}^n\lambda_j^{(i)}h_i$ for each $j\in\{1,\ldots,p\}$. Since $u^*\neq 0$, after reordering, we can take $m\leq p$ for which $g_j\neq 0$ for all $j\leq m$ and $g_j=0$ for all $j>m+1$. Then $\sum_{j=1}^m g_j\boxtimes\phi_j$ is a representation of $u^*$ satisfying the required conditions.
\end{proof}

Our next aim is to show that the associated Lipschitz tensor product $X^\#\boxast E^*$ is linearly isomorphic to the space of Lipschitz finite-rank operators from $X$ to $E^*$. This class of Lipschitz operators appears in \cite{j70, jsv}.

Let us recall that if $X$ is a set and $E$ is a vector space, then a map $f\colon X\to E$ is said to have finite-dimensional rank if the subspace of $E$ generated by $f(X)$, $\lin(f(X))$, is finite-dimensional in whose case the rank of $f$, denoted by $\rank(f)$, is defined as the dimension of $\lin(f(X))$. 

For a pointed metric space $X$ and a Banach space $E$, we denote by $\LipoF(X,E^*)$ the set of all Lipschitz finite-rank operators from $X$ to $E^*$. Clearly, $\LipoF(X,E^*)$ is a linear subspace of $\Lipo(X,E^*)$. For any $g\in X^\#$ and $\phi\in E^*$, we consider in Lemma \ref{lem-3} the elements $g\cdot\phi$ of $\LipoF(X,E^*)$ defined by $(g\cdot\phi)(x)=g(x)\phi$ for all $x\in X$. Note that $\rank(g\cdot\phi)=1$ if $g\neq 0$ and $\phi\neq 0$. Now we prove that these elements generate linearly the space $\LipoF(X,E^*)$.

\begin{lemma}\label{lem-finite}
Every element $f\in\LipoF(X,E^*)$ has a representation in the form $f=\sum_{j=1}^m g_j\cdot \phi_j$, where $m=\rank(f)$, $g_1,\ldots,g_m\in X^\#$ and $\phi_1,\ldots,\phi_m\in E^*$.
\end{lemma}

\begin{proof}
Suppose that $\lin(f(X))$ is $m$-dimensional and let $\{\phi_1,\ldots,\phi_m\}$ be a basis of $\lin(f(X))$. Then, for each $x\in X$, the element $f(x)\in f(X)$ is expressible in a unique form as $f(x)=\sum_{j=1}^m \lambda^{(x)}_j\phi_j$ with $\lambda^{(x)}_1,\ldots,\lambda^{(x)}_m\in\K$. For each $j\in\{1,\ldots,m\}$, define the linear map $y^j\colon\lin(f(X))\to\K$ by $y^j(f(x))=\lambda^{(x)}_j$ for all $x\in X$. Let $g_j=y ^j\circ f$. Clearly, $g_j\in X^\#$ and, given $x\in X$, we have $f(x)=\sum_{j=1}^m \lambda^{(x)}_j\phi_j=\sum_{j=1}^m g_j(x)\phi_j$. Hence $f=\sum_{j=1}^m g_j\cdot\phi_j$.
\end{proof}

\begin{theorem}\label{theo-identification2}
The map $K\colon X^\#\boxast E^*\to\LipoF(X,E^*)$, defined by 
$$
K\left(\sum_{j=1}^m g_j\boxtimes\phi_j\right)=\sum_{j=1}^m g_j\cdot\phi_j,
$$
is a linear isomorphism.
\end{theorem}

\begin{proof}
The map $K$ is well defined by applying Lemma \ref{lem-4.1} and Theorem \ref{theo-dual-pair}. Clearly, $K$ is linear. Moreover, it is surjective by Lemma \ref{lem-finite} and injective by Lemma \ref{lem-4.1}.
\end{proof}

We next introduce the concept of Lipschitz tensor product operator of a Lipschitz operator and a bounded linear operator.

\begin{definition}\label{def-Lipschitz tensor product operator}
Let $X,Y$ be pointed metric spaces and $E,F$ Banach spaces. Let $h\in\Lipo(X,Y)$ and $T\in\L(E;F)$. The map $h\boxtimes T\colon X\boxtimes E\to Y\boxtimes F$, given by  
$$
(h\boxtimes T)(u)=\sum_{i=1}^n\delta_{(h(x_i),h(y_i))}\boxtimes T(e_i)
$$
for $u=\sum_{i=1}^n \delta_{(x_i,y_i)}\boxtimes e_i\in X\boxtimes E$, is called the Lipschitz tensor product operator of $h$ and $T$. 
\end{definition}

\begin{lemma}\label{lem-7}
Let $h\in\Lipo(X,Y)$ and $T\in\L(E;F)$. Then $h\boxtimes T\colon X\boxtimes E\to Y\boxtimes F$ is a well-defined linear operator.
\end{lemma}

\begin{proof}
Let $u=\sum_{i=1}^n \delta_{(x_i,y_i)}\boxtimes e_i$ and $v=\sum_{i=1}^m\delta_{(x'_i,y'_i)}\boxtimes e'_i$ be in $X\boxtimes E$. If $u=v$, then Proposition \ref{pro-0} says us that $\sum_{i=1}^{n}(g(x_i)-g(y_i))e_i=\sum_{i=1}^{m}(g(x'_i)-g(y'_i))e'_i$ for all $g\in B_{X^\#}$. In particular, this holds for all function in $B_{X^\#}$ of the form $(f\circ h)/(1+\Lip(h))$ with $f$ varying in $B_{Y^\#}$. It follows that $\sum_{i=1}^{n}\left(f(h(x_i))-f(h(y_i))\right)T(e_i)=\sum_{i=1}^{m}\left(f(h(x'_i))-f(h(y'_i))\right)T(e'_i)$ for all $f\in B_{Y^\#}$, and this implies that $\sum_{i=1}^n \delta_{(h(x_i),h(y_i))}\boxtimes T(e_i)=\sum_{i=1}^m\delta_{(h(x'_i),h(y'_i))}\boxtimes T(e'_i)$ again by Proposition \ref{pro-0}. Hence the map $h\boxtimes T$ is well defined.

We see that $h\boxtimes T$ is linear. Let $\lambda\in\K$. Then $\lambda u=\sum_{i=1}^n \delta_{(x_i,y_i)}\boxtimes(\lambda e_i)$, and Definition \ref{def-Lipschitz tensor product operator} and Lemma \ref{lem-1} give
\begin{align*}
(h\boxtimes T)(\lambda u)
&=\sum_{i=1}^n\delta_{(h(x_i),h(y_i))}\boxtimes T(\lambda e_i)\\
&=\sum_{i=1}^n\delta_{(h(x_i),h(y_i))}\boxtimes \lambda T(e_i)\\
&=\lambda\sum_{i=1}^n\delta_{(h(x_i),h(y_i))}\boxtimes T (e_i)\\
&=\lambda(h\boxtimes T)(u).
\end{align*}
Take $u+v=\sum_{i=1}^{n+m} \delta_{(x''_i,y''_i)}\boxtimes e''_i$ as in Lemma \ref{lem-conca}. Then we have
\begin{align*}
(h\boxtimes T)(u+v)
&=\sum_{i=1}^{n+m}\delta_{(h(x''_i),h(y''_i))}\boxtimes T(e''_i)\\
&=\sum_{i=1}^{n}\delta_{(h(x''_i),h(y''_i))}\boxtimes T(e''_i)+\sum_{i=n+1}^{n+m}\delta_{(h(x''_i),h(y''_i))}\boxtimes T(e''_i)\\
&=\sum_{i=1}^{n}\delta_{(h(x_i),h(y_i))}\boxtimes T(e_i)+\sum_{i=n+1}^{n+m}\delta_{(h(x'_{i-n}),h(y'_{i-n}))}\boxtimes T(e'_{i-n})\\
&=\sum_{i=1}^{n}\delta_{(h(x_i),h(y_i))}\boxtimes T(e_i)+\sum_{i=1}^{m}\delta_{(h(x'_i),h(y'_i))}\boxtimes T(e'_i)\\
&=(h\boxtimes T)(u)+(h\boxtimes T)(v).
\end{align*}
\end{proof}

\section{Lipschitz cross-norms}\label{3}

We denote the linear space $X\boxtimes E$ endowed with a norm $\alpha$ by $X\boxtimes_\alpha E$, and its completion by $X\widehat{\boxtimes}_\alpha E$. We are looking for a norm on the linear space $X\boxtimes E$, and for our purposes it is convenient to work with norms that satisfy the following conditions.

\begin{definition}\label{def-cross-norm}
Let $X$ be a pointed metric space and $E$ a Banach space. We say that a norm $\alpha$ on $X\boxtimes E$ is a Lipschitz cross-norm if 
$$
\alpha\left(\delta_{(x,y)}\boxtimes e\right)=d(x,y)\left\|e\right\|
$$
for all $(x,y)\in X^2$ and $e\in E$. 

A Lipschitz cross-norm $\alpha$ on $X\boxtimes E$ is said to be dualizable if given $g\in X^\#$ and $\phi\in E^*$, we have
$$
\left|\sum_{i=1}^n\left(g(x_i)-g(y_i)\right)\left\langle \phi,e_i\right\rangle\right|
\leq\Lip(g)\left\|\phi\right\|\alpha\left(\sum_{i=1}^n \delta_{(x_i,y_i)}\boxtimes e_i\right)
$$
for all $\sum_{i=1}^n \delta_{(x_i,y_i)}\boxtimes e_i\in X\boxtimes E$.

A Lipschitz cross-norm $\alpha$ on $X\boxtimes E$ is called uniform if given $h\in\Lipo(X,X)$ and $T\in\L(E;E)$, we have
$$
\alpha\left(\sum_{i=1}^n\delta_{(h(x_i),h(y_i))}\boxtimes T(e_i)\right)
\leq\Lip(h)\left\|T\right\|\alpha\left(\sum_{i=1}^n \delta_{(x_i,y_i)}\boxtimes e_i\right)
$$ 
for all $\sum_{i=1}^n \delta_{(x_i,y_i)}\boxtimes e_i\in X\boxtimes E$.
\end{definition}

The dualizable Lipschitz cross-norms on $X\boxtimes E$ may be characterized by the boundedness of the Lipschitz tensor product functionals.

\begin{proposition}\label{prop-dual}
A Lipschitz cross-norm $\alpha$ on $X\boxtimes E$ is dualizable if and only if, for each $g\in X^\#$ and $\phi\in E^*$, the linear functional $g\boxtimes\phi\colon X\boxtimes_\alpha E\to\mathbb{K}$ is bounded and $\left\|g\boxtimes\phi\right\|=\Lip(g)\left\|\phi\right\|$.
\end{proposition}

\begin{proof}
Let $\alpha$ be a Lipschitz cross-norm on $X\boxtimes E$. Given $g\in X^\#$ and $\phi\in E^*$, if $\alpha$ is dualizable, we have
$$
\left|(g\boxtimes\phi)\left(\sum_{i=1}^n \delta_{(x_i,y_i)}\boxtimes e_i\right)\right|
=\left|\sum_{i=1}^n (g(x_i)-g(y_i))\left\langle\phi,e_i\right\rangle\right|
\leq\Lip(g)\left\|\phi\right\|\alpha\left(\sum_{i=1}^n \delta_{(x_i,y_i)}\boxtimes e_i\right)
$$ 
for all $\sum_{i=1}^n \delta_{(x_i,y_i)}\boxtimes e_i\in X\boxtimes E$. Hence the linear functional $g\boxtimes\phi$ is bounded on $X\boxtimes_\alpha E$ and $\left\|g\boxtimes\phi\right\|\leq\Lip(g)\left\|\phi\right\|$. The opposite inequality $\Lip(g)\left\|\phi\right\|\leq\left\|g\boxtimes\phi\right\|$ is deduced from the fact that 
$$
\left|g(x)-g(y)\right|\left|\left\langle\phi,e\right\rangle\right|
=\left|(g\boxtimes\phi)(\delta_{(x,y)}\boxtimes e)\right|
\leq\left\|g\boxtimes\phi\right\|\alpha\left(\delta_{(x,y)}\boxtimes e\right)
=\left\|g\boxtimes\phi\right\|d(x,y)\left\|e\right\|
$$
for all $x,y\in X$ and $e\in E$.

Conversely, if for any $g\in X^\#$ and $\phi\in E^*$, the linear functional $g\boxtimes\phi\colon X\boxtimes_\alpha E\to\K$ is bounded and $\left\|g\boxtimes \phi\right\|=\Lip(g)\left\|\phi\right\|$, then 
\begin{align*}
\left|\sum_{i=1}^n\left(g(x_i)-g(y_i)\right)\left\langle \phi,e_i\right\rangle\right|
&=\left|(g\boxtimes\phi)\left(\sum_{i=1}^n \delta_{(x_i,y_i)}\boxtimes e_i\right)\right|\\
&\leq\left\|g\boxtimes\phi\right\|\alpha\left(\sum_{i=1}^n \delta_{(x_i,y_i)}\boxtimes e_i\right)\\
&=\Lip(g)\left\|\phi\right\|\alpha\left(\sum_{i=1}^n \delta_{(x_i,y_i)}\boxtimes e_i\right)
\end{align*}
for all $\sum_{i=1}^n \delta_{(x_i,y_i)}\boxtimes e_i\in X\boxtimes E$, and so $\alpha$ is dualizable.
\end{proof}

Similarly, the boundedness of the Lipschitz tensor product operators characterizes the uniform Lipschitz cross-norms on $X\boxtimes E$.

\begin{proposition}\label{prop-dual-2}
A Lipschitz cross-norm $\alpha$ on $X\boxtimes E$ is uniform if and only if, for each $h\in\Lipo(X,X)$ and $T\in\L(E;E)$, the linear operator $h\boxtimes T\colon X\boxtimes_\alpha E\to X\boxtimes_\alpha E$ is bounded and $\left\|h\boxtimes T\right\|=\Lip(h)\left\|T\right\|$.
\end{proposition}

\begin{proof}
Let $\alpha$ be a Lipschitz cross-norm on $X\boxtimes E$. If $\alpha$ is uniform, given $h\in\Lipo(X,X)$ and $T\in\L(E;E)$, we have
$$
\alpha\left((h\boxtimes T)\left(\sum_{i=1}^n \delta_{(x_i,y_i)}\boxtimes e_i\right)\right)
=\alpha\left(\sum_{i=1}^n\delta_{(h(x_i),h(y_i))}\boxtimes T(e_i)\right)
\leq\Lip(h)\left\|T\right\|\alpha\left(\sum_{i=1}^n \delta_{(x_i,y_i)}\boxtimes e_i\right)
$$
for all $\sum_{i=1}^n \delta_{(x_i,y_i)}\boxtimes e_i\in X\boxtimes E$. It follows that the linear operator $h\boxtimes T$ is bounded on $X\boxtimes_\alpha E$ and $\left\|h\boxtimes T\right\|\leq\Lip(h)\left\|T\right\|$. For the reverse inequality, notice that
\begin{align*}
d(h(x),h(y))\left\|T(e)\right\|
&=\alpha\left(\delta_{(h(x),h(y))}\boxtimes T(e)\right)\\
&=\alpha\left((h\boxtimes T)\left(\delta_{(x,y)}\boxtimes e\right)\right)\\
&\leq\left\|h\boxtimes T\right\|\alpha\left(\delta_{(x,y)}\boxtimes e\right)\\
&=\left\|h\boxtimes T\right\|d(x,y)\left\|e\right\|
\end{align*}
for all $x,y\in X$ and $e\in E$, and therefore $\Lip(h)\left\|T\right\|\leq\left\|h\boxtimes T\right\|$.

Conversely, if for each $h\in\Lipo(X,X)$ and $T\in\L(E;E)$, the linear map $h\boxtimes T\colon X\boxtimes_\alpha E\to X\boxtimes_\alpha E$ is bounded and $\left\|h\boxtimes T\right\|=\Lip(h)\left\|T\right\|$, then 
\begin{align*}
\alpha\left(\sum_{i=1}^n\delta_{(h(x_i),h(y_i))}\boxtimes T(e_i)\right)
&=\alpha\left((h\boxtimes T)\left(\sum_{i=1}^n \delta_{(x_i,y_i)}\boxtimes e_i\right)\right)\\
&\leq\left\|h\boxtimes T\right\|\alpha\left(\sum_{i=1}^n \delta_{(x_i,y_i)}\boxtimes e_i\right)\\
&=\Lip(h)\left\|T\right\|\alpha\left(\sum_{i=1}^n \delta_{(x_i,y_i)}\boxtimes e_i\right)
\end{align*}
for all $\sum_{i=1}^n \delta_{(x_i,y_i)}\boxtimes e_i\in X\boxtimes E$, and so $\alpha$ is uniform.
\end{proof}

\begin{remark}\label{prop-dual-3}
A reading of the proofs of the two preceding propositions shows that a Lipschitz cross-norm $\alpha$ on $X\boxtimes E$ is dualizable (uniform) if for each $g\in X^\#$ and $\phi\in E^*$, then $g\boxtimes\phi\in (X\boxtimes_\alpha E)^*$ and $\left\|g\boxtimes\phi\right\|\leq\Lip(g)\left\|\phi\right\|$ (respectively, if for each $h\in\Lipo(X,X)$ and $T\in\L(E;E)$, then $h\boxtimes T\in\L(X\boxtimes_\alpha E;X\boxtimes_\alpha E)$ and $\left\|h\boxtimes T\right\|\leq\Lip(h)\left\|T\right\|$).
\end{remark}

As a consequence of this remark, note that if $\alpha$ is a dualizable Lipschitz cross-norm on $X\boxtimes E$, then $X^\#\boxast_{\alpha'} E^*$ is a linear subspace of $(X\boxtimes_\alpha E)^*$. Next we introduce the concept of Lipschitz cross-norm on $X^\#\boxast E^*$.

\begin{definition}\label{def-Lipschitz-cross-norm for}
Let $X$ be a pointed metric space and $E$ a Banach space. We say that a norm $\beta$ on $X^\#\boxast E^*$ is a Lipschitz cross-norm if $\beta(g\boxtimes \phi)=\Lip(g)\left\|\phi\right\|$ for all $g\in X^\#$ and $\phi\in E^*$. 
\end{definition}

As a consequence of Proposition \ref{prop-dual}, we have the following Lipschitz cross-norm on $X^\#\boxast E^*$.

\begin{corollary}\label{pro-dual-norm}
Let $\alpha$ be a dualizable Lipschitz cross-norm on $X\boxtimes E$. The restriction to $X^\#\boxast E^*$ of the canonical norm of $(X\boxtimes_\alpha E)^*$, that is, the map $\alpha'\colon X^\#\boxast E^*\to\R$, given by   
$$
\alpha'(u^*)
=\sup\left\{\left|\left(\sum_{j=1}^m g_j\boxtimes\phi_j\right)\left(\sum_{i=1}^n \delta_{(x_i,y_i)}\boxtimes e_i\right)\right|
\colon \alpha\left(\sum_{i=1}^n \delta_{(x_i,y_i)}\boxtimes e_i\right)\leq 1\right\}
$$
for $u^*=\sum_{j=1}^m g_j\boxtimes \phi_j\in X^\#\boxast E^*$, is a Lipschitz cross-norm on $X^\#\boxast E^*$. 
\end{corollary}

\begin{definition}\label{def-Lipschitz-associated-cross-norm}
Let $X$ be a pointed metric space and $E$ a Banach space. Let $\alpha$ be a dualizable Lipschitz cross-norm on $X\boxtimes E$. The norm $\alpha'$ on $X^\#\boxast E^*$ is called the associated Lipschitz norm of $\alpha$. The vector space $X^\#\boxast E^*$ with the norm $\alpha'$ will be denoted by $X^\#\boxast_{\alpha'} E^*$ and its completion by $X^\#\widehat{\boxast}_{\alpha'} E^*$.
\end{definition}

\section{The induced Lipschitz dual norm}\label{4}

\begin{definition}\label{def-L}
For each $u=\sum_{i=1}^n \delta_{(x_i,y_i)}\boxtimes e_i\in X\boxtimes E$, define:
$$
L(u)
=\sup\left\{\left|\sum_{i=1}^n \left\langle f(x_i)-f(y_i),e_i\right\rangle\right|\colon f\in\Lipo(X,E^*),\ \Lip(f)\leq 1 \right\}.
$$
\end{definition}

Note that the supremum on the right side above exists and $L(u)\leq\sum_{i=1}^n d(x_i,y_i)\left\|e_i\right\|$ because
\begin{align*}
\left|\sum_{i=1}^n \left\langle f(x_i)-f(y_i),e_i\right\rangle\right|
&\leq\sum_{i=1}^n \left|\left\langle f(x_i)-f(y_i),e_i\right\rangle\right|\\
&\leq\sum_{i=1}^n \left\|f(x_i)-f(y_i)\right\|\left\|e_i\right\|\\
&\leq\Lip(f)\sum_{i=1}^n d(x_i,y_i)\left\|e_i\right\|
\end{align*}
for all $f\in\Lipo(X,E^*)$. Moreover, $L$ defines a map from $X\boxtimes E$ to $\R$ by Lemma \ref{lem-2}. 

\begin{theorem}\label{teo-L}
The linear space $X\boxtimes E$ is contained in $\Lipo(X,E^*)^*$ and $L$ is the dual norm of the norm $\Lip$ of $\Lipo(X,E^*)$ induced on $X\boxtimes E$. Moreover, $L$ is a Lipschitz cross-norm on $X\boxtimes E$.
\end{theorem}

\begin{proof}
Let $x,y\in X$ and $e\in E$. Since $\delta_{(x,y)}\boxtimes e$ is a linear map on $\Lipo(X,E^*)$ and 
$$
\left|(\delta_{(x,y)}\boxtimes e)(f)\right|
=\left|\left\langle f(x)-f(y),e\right\rangle\right|
\leq\left\|f(x)-f(y)\right\|\left\|e\right\|
\leq\Lip(f)d(x,y)\left\|e\right\|
$$
for all $f\in\Lipo(X,E^*)$, then $\delta_{(x,y)}\boxtimes e\in\Lipo(X,E^*)^*$ and thus $X\boxtimes E\subset\Lipo(X,E^*)^*$. For every $u=\sum_{i=1}^n \delta_{(x_i,y_i)}\boxtimes e_i\in X\boxtimes E$, we have
$$
L(u)=\sup\left\{\left|u(f)\right|\colon f\in\Lipo(X,E^*),\ \Lip(f)\leq 1 \right\}
$$
by Lemma \ref{lem-2}, and therefore $L$ is the dual norm of the norm $\Lip$ of $\Lipo(X,E^*)$ induced on $X\boxtimes E$. Finally, we prove that $L$ is a Lipschitz cross-norm. By above-proved, $\left|(\delta_{(x,y)}\boxtimes e)(f)\right|\leq d(x,y)\left\|e\right\|$ for all $f\in\Lipo(X,E^*)$ with $\Lip(f)\leq 1$, and hence 
$L(\delta_{(x,y)}\boxtimes e)\leq d(x,y)\left\|e\right\|$. For the reverse estimate, take $\phi\in E^*$ with $\left\|\phi\right\|=1$ satisfying $\left|\left\langle\phi,e\right\rangle\right|=\left\|e\right\|$, and consider the map $f\colon X\to E^*$ given by 
$$
f(z)=(d(0,x)-d(z,x))\phi\qquad (z\in X).
$$
An easy verification shows that $f$ is in $\Lipo(X,E^*)$ with $\Lip(f)\leq 1$ and    
$$
\left|(\delta_{(x,y)}\boxtimes e)(f)\right|
=\left|\left\langle f(x)-f(y),e\right\rangle\right|
=\left|\left\langle d(x,y)\phi,e\right\rangle\right|
=d(x,y)\left|\left\langle \phi,e\right\rangle\right|
=d(x,y)\left\|e\right\|,
$$
and therefore  
$
d(x,y)\left\|e\right\|
=\left|(\delta_{(x,y)}\boxtimes e)(f)\right|
\leq L(\delta_{(x,y)}\boxtimes e)$.
\end{proof}

The following result is essentially known. For completeness we include it here with an alternate proof.

\begin{theorem}\label{theo-dual-classic}\cite[Theorem 4.1]{j70}
Let $X$ be a pointed metric space and let $E$ be a Banach space. Then $\Lipo(X,E^*)$ is isometrically isomorphic to $(X\widehat{\boxtimes}_L E)^*$, via the map $\Lambda_0\colon\Lipo(X,E^*)\to (X\widehat{\boxtimes}_L E)^*$ given by
$$
\Lambda_0(f)(u)
=\sum_{i=1}^n\left\langle f(x_i)-f(y_i),e_i\right\rangle
$$
for $f\in\Lipo(X,E^*)$ and $u=\sum_{i=1}^n \delta_{(x_i,y_i)}\boxtimes e_i\in X\boxtimes_L E$. Its inverse $\Lambda_0^{-1}\colon (X\widehat{\boxtimes}_L E)^*\to \Lipo(X,E^*)$ is defined by 
$$
\left\langle \Lambda_0^{-1}(\varphi)(x),e\right\rangle
=\left\langle\varphi,\delta_{(x,0)}\boxtimes e\right\rangle 
$$
for $\varphi\in (X\widehat{\boxtimes}_L E)^*$, $x\in X$ and $e\in E$.
\end{theorem}

\begin{proof}
Let $f\in\Lipo(X,E^*)$ and let $\Lambda(f)$ be the linear functional on $X\boxtimes E$ defined in Corollary \ref{linearization}. Notice that $\Lambda(f)\in (X\boxtimes_L E)^*$ and $\left\|\Lambda(f)\right\|\leq\Lip(f)$ since   
$$
\left|\Lambda(f)(u)\right|=\left|\sum_{i=1}^n\left\langle f(x_i)-f(y_i), e_i\right\rangle\right|\leq\Lip(f)L(u)
$$
for all $u=\sum_{i=1}^n \delta_{(x_i,y_i)}\boxtimes e_i\in X\boxtimes E$. By the denseness of $X\boxtimes_L E$ in $X\widehat{\boxtimes}_L E$, there is a unique continuous extension $\Lambda_0(f)$ of $\Lambda(f)$ to $X\widehat{\boxtimes}_L E$. Let $\Lambda_0\colon\Lipo(X,E^*)\to(X\widehat{\boxtimes}_L E)^*$ be the map so defined. Since $\Lambda\colon\Lipo(X,E^*)\to (X\boxtimes E)'$ is a linear monomorphism by Corollary \ref{linearization}, it follows easily that so is also $\Lambda_0$. 


In order to see that $\Lambda_0$ is a surjective isometry, let $\varphi$ be an element of $ (X\widehat{\boxtimes}_L E)^*$. Define $f\colon X\to E^*$ by
$$
\left\langle f(x),e\right\rangle=\varphi(\delta_{(x,0)}\boxtimes e)
\qquad\left(x\in X,\; e\in E\right).
$$ 
It is plain that $f(x)$ is a well-defined bounded linear functional on $E$ and that $f$ is well defined. Observe that $\left\langle f(x)-f(y),e\right\rangle=\varphi(\delta_{(x,y)}\boxtimes e)$ for all $x,y\in X$ and $e\in E$. Fix $x,y\in X$. It follows that 
$$
\left|\left\langle f(x)-f(y),e\right\rangle\right|
=\left|\varphi(\delta_{(x,y)}\boxtimes e)\right|
\leq\left\|\varphi\right\| L(\delta_{(x,y)}\boxtimes e)
=\left\|\varphi\right\| d(x,y)\left\|e\right\|
$$
for all $e\in E$, and so $\left\|f(x)-f(y)\right\|\leq \left\|\varphi\right\| d(x,y)$. Hence $f\in\Lipo(X,E^*)$ and $\Lip(f)\leq\left\|\varphi\right\|$. For any $u=\sum_{i=1}^n \delta_{(x_i,y_i)}\boxtimes e_i\in X\boxtimes_L E$, we get 
$$
\Lambda_0(f)(u)
=\sum_{i=1}^n\left\langle f(x_i)-f(y_i),e_i\right\rangle
=\sum_{i=1}^n \varphi(\delta_{(x_i,y_i)}\boxtimes e_i)
=\varphi\left(\sum_{i=1}^n \delta_{(x_i,y_i)}\boxtimes e_i\right)
=\varphi(u).
$$
Hence $\Lambda_0(f)=\varphi$ on a dense subspace of $X\widehat{\boxtimes}_L E$ and, consequently, $\Lambda_0(f)=\varphi$. Moreover, $\Lip(f)\leq\left\|\varphi\right\|=\left\|\Lambda_0(f)\right\|$. This completes the proof of the theorem. 
\end{proof} 

\section{The Lipschitz injective norm}\label{5}

We introduce the Lipschitz injective norm on $X\boxtimes E$.

\begin{definition}\label{def-inj-norm}
For each $u=\sum_{i=1}^n \delta_{(x_i,y_i)}\boxtimes e_i\in X\boxtimes E$, define:
$$
\varepsilon(u)
=\sup\left\{\left|\sum_{i=1}^n\left(g(x_i)-g(y_i)\right)\left\langle\phi,e_i\right\rangle\right|
\colon g\in B_{X^\#}, \; \phi\in B_{E^*}\right\}.
$$
\end{definition}

Notice that the supremum on the right side in the previous definition exists since
\begin{align*}
\left|\sum_{i=1}^n\left(g(x_i)-g(y_i)\right)\left\langle \phi,e_i\right\rangle\right|
&\leq\sum_{i=1}^n \left|\left(g(x_i)-g(y_i)\right)\left\langle \phi,e_i\right\rangle\right|\\
&\leq\sum_{i=1}^n \Lip(g)d(x_i,y_i)\left\|\phi\right\|\left\|e_i\right\|\\
&\leq\sum_{i=1}^n d(x_i,y_i)\left\|e_i\right\|
\end{align*}
for all $g\in B_{X^\#}$ and $\phi\in B_{E^*}$. Note that 
$$
\sum_{i=1}^n\left(g(x_i)-g(y_i)\right)\left\langle\phi,e_i\right\rangle=(g\boxtimes\phi)\left(\sum_{i=1}^n \delta_{(x_i,y_i)}\boxtimes e_i\right)=(g\boxtimes\phi)(u),
$$
and, consequently, $\varepsilon(u)$ does not depend on the representation of $u$ by Lemma \ref{lem-4}, so $\varepsilon$ defines a map from $X\boxtimes E$ to $\R$ .

\begin{theorem}\label{teo-inj-norm}
$\varepsilon$ is a uniform and dualizable Lipschitz cross-norm on $X\boxtimes E$.
\end{theorem}

\begin{proof}
Let $u=\sum_{i=1}^n \delta_{(x_i,y_i)}\boxtimes e_i\in X\boxtimes E$. Suppose that $\varepsilon(u)=0$. Then $\sum_{i=1}^n(g(x_i)-g(y_i))\langle\phi,e_i\rangle=0$ for all $g\in B_{X^\#}$ and $\phi\in B_{E^*}$, and this happens if and only if $u=0$ by Proposition \ref{pro-0}.

For any $\lambda\in\K$, we have 
$$
\varepsilon(\lambda u)
=\sup\left\{\left|(g\boxtimes\phi)(\lambda u)\right|\colon g\in B_{X^\#}, \; \phi\in B_{E^*}\right\}
=\left|\lambda|\right|\sup\left\{\left|(g\boxtimes\phi)(u)\right|\colon g\in B_{X^\#}, \; \phi\in B_{E^*}\right\}
=\left|\lambda\right|\varepsilon(u).
$$
Given $v\in X\boxtimes E$, for any $g\in B_{X^\#}$ and $\phi\in B_{E^*}$, it holds that  
$$
\left|(g\boxtimes\phi)(u+v)\right|
\leq\left|(g\boxtimes\phi)(u)\right|+\left|(g\boxtimes\phi)(v)\right|
\leq\varepsilon(u)+\varepsilon(v),
$$
and therefore $\varepsilon(u+v)\leq\varepsilon(u)+\varepsilon(v)$. Hence $\varepsilon$ is a norm on $X\boxtimes E$. 

We claim that $\varepsilon$ is a Lipschitz cross-norm. Take $\delta_{(x,y)}\boxtimes e\in X\boxtimes E$. For any $g\in B_{X^\#}$ and $\phi\in B_{E^*}$, we have
$$
\left|\left(g(x)-g(y)\right)\left\langle\phi,e\right\rangle\right|\leq \Lip(g)d(x,y)\left\|\phi\right\|\left\|e\right\|\leq d(x,y)\left\|e\right\|,
$$
and so $\varepsilon(\delta_{(x,y)}\boxtimes e)\leq d(x,y)\left\|e\right\|$. For the converse inequality, we can find $g_0\in B_{X^\#}$ and $\phi_0\in B_{E^*}$ such that $\left|g_0(x)-g_0(y)\right|=d(x,y)$ and $\langle\phi_0,e\rangle=\left\|e\right\|$. For example, $g_0(z)=d(0,x)-d(z,x)$ for all $z\in X$. Then 
$$
\varepsilon\left(\delta_{(x,y)}\boxtimes e\right)\geq\left|\left(g_0(x)-g_0(y)\right)\left\langle\phi_0,e\right\rangle\right|=d(x,y)\left\|e\right\|,
$$
and this proves our claim.

Now take $g\in X^\#$ and $\phi\in E^*$. By Definition \ref{def-inj-norm}, we have
$$
\left|\sum_{i=1}^n\left(g(x_i)-g(y_i)\right)\left\langle\phi,e_i\right\rangle\right|
\leq\Lip(g)\left\|\phi\right\|\varepsilon\left(\sum_{i=1}^n\delta_{(x_i,y_i)}\boxtimes e_i\right)
$$
for all $\sum_{i=1}^n \delta_{(x_i,y_i)}\boxtimes e_i\in X\boxtimes E$. Hence the norm $\varepsilon$ is dualizable. Then, by Corollary \ref{pro-dual-norm}, $\varepsilon'$ is a Lipschitz cross-norm on $X^\#\boxast E^*$.

Finally, we prove that the norm $\varepsilon$ is uniform. Let $h\in\Lipo(X,X)$ and $T\in\L(E;E)$. Let $T^*$ denote the adjoint operator of $T$. We now recall that the Lipschitz adjoint map $h^\#\colon X^\#\to X^\#$, given by $h^\#(g)=g\circ h$ for all $g\in X^\#$, is a continuous linear operator and $\left\|h^\#\right\|=\Lip(h)$. 
Indeed, it is clear that $h^\#$ is linear. Let $g\in X^\#$ and $x,y\in X$. We have
$$
\left|h^\#(g)(x)-h^\#(g)(y)\right|
=\left|g(h(x))-g(h(y))\right|
\leq\Lip(g)d(h(x),h(y))
\leq\Lip(g)\Lip(h)d(x,y),
$$
hence $h^\#(g)\in X^\#$ and $\Lip(h^\#(g))\leq\Lip(g)\Lip(h)$. It follows that $h^\#$ is bounded and $\left\|h^\#\right\|\leq\Lip(h)$. Taking the function defined on $X$ by $g(z)=d(h(x),0)-d(h(x),z)$ which is in $B_{X^\#}$, we get 
$$
d(h(x),h(y))=\left|g(h(x))-g(h(y))\right|=\left|h^\#(g)(x)-h^\#(g)(y)\right|\leq\Lip(h^\#(g))d(x,y),
$$
which gives $\Lip(h)\leq\Lip(h^\#(g))\leq\left\|h^\#\right\|\Lip(g)\leq\left\|h^\#\right\|$ and so $\left\|h^\#\right\|=\Lip(h)$.

Given $\sum_{i=1}^n\delta_{(x_i,y_i)}\boxtimes e_i\in X\boxtimes E$, we have  
\begin{align*}
\left|\sum_{i=1}^n\left(g(h(x_i))-g(h(y_i))\right)\left\langle\phi,T(e_i)\right\rangle\right|
&=\left|(h^\#(g)\boxtimes T^*(\phi))\left(\sum_{i=1}^n\delta_{(x_i,y_i)}\boxtimes e_i\right)\right|\\
&\leq\varepsilon'\left(h^\#(g)\boxtimes T^*(\phi)\right)\varepsilon\left(\sum_{i=1}^n\delta_{(x_i,y_i)}\boxtimes e_i\right)\\
&=\Lip(h^\#(g))\left\|T^*(\phi)\right\|\varepsilon\left(\sum_{i=1}^n\delta_{(x_i,y_i)}\boxtimes e_i\right)\\
&\leq\Lip(h)\Lip(g)\left\|T\right\|\left\|\phi\right\|\varepsilon\left(\sum_{i=1}^n\delta_{(x_i,y_i)}\boxtimes e_i\right)\\
&\leq\Lip(h)\left\|T\right\|\varepsilon\left(\sum_{i=1}^n\delta_{(x_i,y_i)}\boxtimes e_i\right)
\end{align*}
for all $g\in B_{X^\#}$ and $\phi\in B_{E^*}$, and hence
$$
\varepsilon\left(\sum_{i=1}^n\delta_{(h(x_i),h(y_i))}\boxtimes T(e_i)\right)
\leq\Lip(h)\left\|T\right\|\varepsilon\left(\sum_{i=1}^n\delta_{(x_i,y_i)}\boxtimes e_i\right),
$$
which proves that the norm $\varepsilon$ is uniform.
\end{proof}

\begin{theorem}\label{teo-least}
$\varepsilon$ is the least dualizable Lipschitz cross-norm on $X\boxtimes E$.
\end{theorem}

\begin{proof}
According to Theorem \ref{teo-inj-norm}, $\varepsilon$ is a dualizable Lipschitz cross-norm on $X\boxtimes E$. Let $\alpha$ be a dualizable Lipschitz cross-norm on $X\boxtimes E$ and assume, for contradiction, that 
$$
\alpha\left(\sum_{i=1}^n \delta_{(x_i,y_i)}\boxtimes e_i\right)
<\varepsilon\left(\sum_{i=1}^n \delta_{(x_i,y_i)}\boxtimes e_i\right)
$$
for some $\sum_{i=1}^n \delta_{(x_i,y_i)}\boxtimes e_i\in X\boxtimes E$. By the definition of $\varepsilon$, there exist $g\in B_{X^\#}$ and $\phi\in B_{E^*}$ such that 
$$
\alpha\left(\sum_{i=1}^n \delta_{(x_i,y_i)}\boxtimes e_i\right)
<\left|\sum_{i=1}^n\left(g(x_i)-g(y_i)\right)\left\langle\phi,e_i\right\rangle\right|.
$$
By Corollary \ref{pro-dual-norm}, $\alpha'$ is a Lipschitz cross-norm on $X^{\#}\boxast E^*$, and we have 
$$
\left|\sum_{i=1}^n\left(g(x_i)-g(y_i)\right)\left\langle\phi,e_i\right\rangle\right|
=\left|(g\boxtimes\phi)\left(\sum_{i=1}^n \delta_{(x_i,y_i)}\boxtimes e_i\right)\right|
\leq\alpha'(g\boxtimes\phi)\alpha\left(\sum_{i=1}^n \delta_{(x_i,y_i)}\boxtimes e_i\right).
$$
Hence $\alpha'(g\boxtimes\phi)>1$ and thus $\Lip(g)\left\|\phi\right\|<\alpha'(g\boxtimes\phi)$. This contradicts that $\alpha'$ is a Lipschitz cross-norm. Therefore $\alpha\geq\varepsilon$ and this proves the theorem.
\end{proof}

The completion $X\widehat{\boxtimes}_{\varepsilon} E$ of $X\boxtimes_{\varepsilon} E$ is called the injective Lipschitz tensor product of $X$ and $E$.  Next we justify this terminology in the case $\mathbb{K}=\mathbb{R}$. 
\begin{theorem}
Let $X$ be a pointed metric space and let $E$ be a Banach space over $\mathbb{R}$. Let $X_0 \subset X$ be a subset of $X$ containing $0$, and let $E_0$ be a closed linear subspace of $E$.
Then $X_0\widehat{\boxtimes}_{\varepsilon}E_0$ is a 
linear subspace of $X\widehat{\boxtimes}_{\varepsilon}E$.
\end{theorem}

\begin{proof}
Let $u=\sum_{i=1}^n \delta_{(x_i,y_i)}\boxtimes e_i\in X_0\boxtimes E_0$. Note that $u$ can be considered as an element of $X\boxtimes E$ because of Lemma \ref{lem-7}, which says that the Lipschitz tensor product of the two embeddings is well defined. It is sufficient to prove that $\varepsilon_{X\boxtimes E}(u)=\varepsilon_{X_0\boxtimes E_0}(u)$, where 
\begin{align*} 
\varepsilon_{X_0\boxtimes E_0}(u)&=\sup\left\{\left|\sum_{i=1}^n\left(g_0(x_i)-g_0(y_i)\right)\left\langle\phi_0,e_i\right\rangle\right|\colon g_0\in B_{X_0^\#}, \; \phi_0\in B_{E_0^*}\right\},\\
\varepsilon_{X\boxtimes E}(u)&=\sup\left\{\left|\sum_{i=1}^n\left(g(x_i)-g(y_i)\right)\left\langle\phi,e_i\right\rangle\right|\colon g\in B_{X^\#}, \; \phi\in B_{E^*}\right\}.
\end{align*}
By applying the classical Hahn--Banach theorem we can extend each $\phi_0\in B_{E_0^*}$ to a $\phi\in B_{E^*}$,
and by applying the nonlinear Hahn--Banach theorem we can extend each $g_0\in  B_{X_0^\#}$ to a $g\in  B_{X^\#}$,
hence we see that $\varepsilon_{X_0\boxtimes E_0}(u)\leq\varepsilon_{X\boxtimes E}(u)$. Conversely, by restricting the functionals $\phi\in B_{E^*}$ to $E_0$
and the Lipschitz functions $g \in  B_{X^\#}$ to $X_0$, we obtain that $\varepsilon_{X\boxtimes E}(u)\leq\varepsilon_{X_0\boxtimes E_0}(u)$.
\end{proof}

We can identify $X\widehat{\boxtimes}_{\varepsilon} E$ with the space of all approximable bounded linear operators of $(X^\#,\tau_p)$ to $E$.

\begin{proposition}\label{prop-identification2}
The map $J\colon X\boxtimes_{\varepsilon} E\to\mathcal{F}((X^\#,\tau_p);E)$, defined by 
$$
J(u)(g)=\sum_{i=1}^n\left(g(x_i)-g(y_i)\right)e_i
$$
for $u=\sum_{i=1}^n \delta_{(x_i,y_i)}\boxtimes e_i\in X\boxtimes E$ and $g\in X^\#$, is an isometric isomorphism. As a consequence,  $X\widehat{\boxtimes}_{\varepsilon} E$ is isometrically isomorphic to the closure in the operator norm topology of $\mathcal{F}((X^\#,\tau_p);E)$.
\end{proposition}

\begin{proof}
By Theorem \ref{teo-identification}, $J$ is a linear bijection. If $u=\sum_{i=1}^n \delta_{(x_i,y_i)}\boxtimes e_i\in X\boxtimes E$, we have
\begin{align*} 
\left\|J(u)\right\|
&=\sup\left\{\left\|J(u)(g)\right\|\colon g\in B_{X^\#}\right\}\\
&=\sup\left\{\left|\phi\left(\sum_{i=1}^n\left(g(x_i)-g(y_i)\right)e_i\right)\right|\colon g\in B_{X^\#},\ \phi\in B_{E^*}\right\}\\
&=\sup\left\{\left|\sum_{i=1}^n\left(g(x_i)-g(y_i)\right)\left\langle\phi,e_i\right\rangle\right|\colon g\in B_{X^\#},\ \phi\in B_{E^*}\right\}\\
&=\varepsilon(u).
\end{align*}
The consequence is immediate.
\end{proof}

Let $\F(X)$ be the Lipschitz-free Banach space over a pointed metric space $X$. Let us recall that $\F(X)$ is the closed linear subspace of $(X^\#)^*$ spanned by the set $\{\delta_x\colon x\in X\}$, where for each $x\in X$, $\delta_x$ is the evaluation functional at the point $x$ defined on $X^\#$.
Combining Proposition \ref{prop-identification2} and Corollary \ref{coro-teo-identification}, we can prove that if
$X$ is a pointed metric space, then $X\widehat{\boxtimes}_{\varepsilon}\K$ is isometrically isomorphic to $\F(X)$;
in fact, much more is true.
We show below that the space $X\widehat{\boxtimes}_{\varepsilon}E$ can be identified with the injective Banach-space tensor product $\F(X) \widehat{\otimes}_\varepsilon E$.
First, let us recall some fundamental properties of the space $\F(X)$.

\begin{theorem}\label{thm-Arens-Eells}\cite{ae},\cite[pp. 39-41]{w}
Let $X$, $Y$ be pointed metric spaces, and $E$ a Banach space.
\begin{enumerate}
\item The dual of $\F(X)$ is (canonically) isometrically isomorphic to $X^\#$, with the duality pairing given by $\pair{g}{\delta_x}=g(x)$ for all $g\in X^\#$ and $x \in X$. Moreover, on bounded subsets of $X^\#$, the weak* topology coincides with the topology of pointwise convergence.
\item The map $\iota_X\colon x\mapsto\delta_x$ is an isometric embedding of $X$ into $\F(X)$.
\item For any Lipschitz map $T \colon  X \to Y$ with $T(0)=0$, there is a unique linear map $\tilde{T}\colon\F(X)\to\F(Y)$ such that $\tilde{T}\circ\iota_X=\iota_y\circ T$. Furthermore, $\bign{\tilde{T}}=\Lip(T)$.
\item For any Lipschitz map $T \colon  X \to E$ with $T(0)=0$, there is a unique linear map $\hat{T}\colon\F(X)\to E$ such that $\hat{T}\circ\iota_X=T$. Furthermore, $\bign{\hat{T}}=\Lip(T)$.
\end{enumerate}
\end{theorem}

It is because of the universal properties above that the space $\F(X)$ is called the Lipschitz-free space over $X$, or simply the free space over $X$.
These spaces have been recently used as tools in nonlinear Banach space theory, see \cite{gk,k04} and the survey \cite{glz}.

\begin{proposition}\label{prop-Lipschitz-injective-is-Banach-injective}
The map $I \colon  X\boxtimes_{\varepsilon} E\to \F(X) \otimes_\varepsilon E $, defined by 
$$
I(u) = \sum_{i=1}^n(\delta_{x_i}-\delta_{y_i}) \otimes e_i 
$$
for $u=\sum_{i=1}^n \delta_{(x_i,y_i)}\boxtimes e_i\in X\boxtimes E$, is a linear isometry. As a consequence,  $X\widehat{\boxtimes}_{\varepsilon} E$ is isometrically isomorphic to $\F(X) \widehat{\otimes}_\varepsilon E$.
\end{proposition}

\begin{proof}
Let $u=\sum_{i=1}^n \delta_{(x_i,y_i)}\boxtimes e_i\in X\boxtimes E$. Since $\F(X)^*\equiv X^\#$, note that the norm of $\sum_{i=1}^n(\delta_{x_i}-\delta_{y_i}) \otimes e_i $ in $\F(X) \otimes_\varepsilon E$ is given by
$$
\sup\left\{\left|\sum_{i=1}^n\pair{g}{\delta_{x_i}-\delta_{y_i}} \left\langle\phi,e_i\right\rangle\right|
\colon g\in B_{X^\#}, \; \phi\in B_{E^*}\right\}.
$$
Since $\pair{g}{\delta_{x_i}-\delta_{y_i}}$ is precisely $g(x_i)-g(y_i)$, Proposition \ref{pro-0} shows that $I$ is well defined (and thus linear) and moreover a quick glance at 
Definition \ref{def-inj-norm} shows that $I$ is an isometry.
	
Recall that the linear span of $\{\delta_x\}_{x\in X}$ is dense in $\F(X)$, hence the tensors of the form $\sum_{i=1}^n(\delta_{x_i}-\delta_{y_i})\otimes e_i$, with $x_i,y_i \in X$ and $e_i \in E$, are dense in $\F(X) \widehat{\otimes}_\varepsilon E$.	This shows that the map $I$ has dense range, and thus $X\widehat{\boxtimes}_{\varepsilon} E$ is isometrically isomorphic to $\F(X) \widehat{\otimes}_\varepsilon E$.
\end{proof}

\section{The Lipschitz projective norm}\label{6}

We introduce the Lipschitz projective norm on $X \boxtimes E$.

\begin{definition}\label{def-pro-norm}
For each $u\in X\boxtimes E$, define:
$$
\pi(u)=\inf\left\{\sum_{i=1}^n d(x_i,y_i)\left\|e_i\right\|\colon u=\sum_{i=1}^n\delta_{(x_i,y_i)}\boxtimes e_i\right\},
$$
the infimum being taken over all representations of $u$.
\end{definition}


\begin{theorem}\label{teo-pro-norm}
$\pi$ is a uniform and dualizable Lipschitz cross-norm on $X\boxtimes E$ such that $L\leq\pi$.
\end{theorem}

\begin{proof}
Let $u\in X\boxtimes E$ and let $\sum_{i=1}^n\delta_{(x_i,y_i)}\boxtimes e_i$ be a representation of $u$. Using Definition \ref{def-L}, we have seen that $L(u)\leq\sum_{i=1}^n d(x_i,y_i)\left\|e_i\right\|$. Since this holds for every representation of $u$, it follows that $L(u)\leq\pi(u)$. 
Suppose that $\pi(u)=0$. Since $L(u)\leq\pi(u)$ and $L$ is a norm on $X\boxtimes E$, then $u=0$.

We check that $\pi(\lambda u)=\left|\lambda\right|\pi(u)$. If $\lambda\in\K$, then  $\lambda u=\sum_{i=1}^n \delta_{(x_i,y_i)}\boxtimes (\lambda e_i)$ and so 
$$
\pi(\lambda u)\leq\sum_{i=1}^n d(x_i,y_i)\left\|\lambda e_i\right\|=\left|\lambda\right|\sum_{i=1}^n d(x_i,y_i)\left\|e_i\right\|.
$$
Since the representation of $u$ is arbitrary, this implies that $\pi(\lambda u)\leq\left|\lambda\right|\pi(u)$. If $\lambda=0$, we have $\pi(\lambda u)=0=\left|\lambda\right|\pi(u)$ since $\pi(u)\geq 0$ for all $u\in X\boxtimes E$. Assume that $\lambda\neq 0$. Similarly, we have $\pi(u)=\pi(\lambda^{-1}(\lambda u))\leq\left|\lambda^{-1}\right|\pi(\lambda u)$, thus $\left|\lambda\right|\pi(u)\leq \pi(\lambda u)$ and hence $\pi(\lambda u)=\left|\lambda\right|\pi(u)$.

We show that $\pi(u+v)\leq\pi(u)+\pi(v)$ for all $u,v\in X\boxtimes E$. Let $\varepsilon>0$. Then there are representations $u=\sum_{i=1}^n \delta_{(x_i,y_i)}\boxtimes e_i$ and $v=\sum_{i=1}^m \delta_{(x'_i,y'_i)}\boxtimes e'_i$ such that $\sum_{i=1}^n d(x_i,y_i)\left\|e_i\right\|<\pi(u)+\varepsilon/2$ and $\sum_{i=1}^m d(x'_i,y'_i)\left\|e'_i\right\|<\pi(v)+\varepsilon/2$. We can concatenate these representations to get a representation $\sum_{i=1}^{n+m} \delta_{(x''_i,y''_i)}\boxtimes e''_i$ for $u+v$ as in Lemma \ref{lem-conca}. By Definition \ref{def-pro-norm}, it follows that 
\begin{align*}
\pi(u+v)&\leq\sum_{i=1}^{n+m} d(x''_i,y''_i)\left\|e''_i\right\|\\
&=\sum_{i=1}^{n}d(x''_i,y''_i)\left\|e''_i\right\|+\sum_{i=n+1}^{n+m}d(x''_i,y''_i)\left\|e''_i\right\|\\
&=\sum_{i=1}^{n}d(x_i,y_i)\left\|e_i\right\|+\sum_{i=n+1}^{n+m}d(x'_{i-n},y'_{i-n})\left\|e'_{i-n}\right\|\\
&=\sum_{i=1}^n d(x_i,y_i)\left\|e_i\right\|+\sum_{i=1}^m d(x'_i,y'_i)\left\|e'_i\right\|\\
&<\pi(u)+\pi(v)+\varepsilon .
\end{align*}
By the arbitrariness of $\varepsilon$, we deduce that $\pi(u+v)\leq\pi(u)+\pi(v)$. Hence $\pi$ is a norm on $X\boxtimes E$. 

We now prove that $\pi$ is a Lipschitz cross-norm. Let $(x,y)\in X^2$ and $e\in E$. It is immediate that $\pi(\delta_{(x,y)}\boxtimes e)\leq d(x,y)\left\|e\right\|$. Conversely, since $L$ is a Lipschitz cross-norm on $X\boxtimes E$ and $L\leq\pi$, it follows that $d(x,y)\left\|e\right\|=L(\delta_{(x,y)}\boxtimes e)\leq\pi(\delta_{(x,y)}\boxtimes e)$.

Let $g\in X^\#$ and $\phi\in E^*$. For any $\sum_{i=1}^n \delta_{(x_i,y_i)}\boxtimes e_i\in X\boxtimes E$, we have
$$
\left|(g\boxtimes\phi)\left(\sum_{i=1}^n \delta_{(x_i,y_i)}\boxtimes e_i\right)\right|
=\left|\sum_{i=1}^n\left(g(x_i)-g(y_i)\right)\left\langle\phi,e_i\right\rangle\right|
\leq\Lip(g)\left\|\phi\right\|\sum_{i=1}^n d(x_i,y_i)\left\|e_i\right\|.
$$
Since the value of $(g\boxtimes\phi)\left(\sum_{i=1}^n \delta_{(x_i,y_i)}\boxtimes e_i\right)$ does not depend on the representation of $\sum_{i=1}^n \delta_{(x_i,y_i)}\boxtimes e_i$ by Lemma \ref{lem-4}, 
it follows that 
$$
\left|(g\boxtimes\phi)\left(\sum_{i=1}^n \delta_{(x_i,y_i)}\boxtimes e_i\right)\right|
\leq\Lip(g)\left\|\phi\right\|\pi\left(\sum_{i=1}^n\delta_{(x_i,y_i)}\boxtimes e_i\right). 
$$
Therefore the Lipschitz cross-norm $\pi$ is dualizable by Proposition \ref{prop-dual} and Remark \ref{prop-dual-3}.

Similarly, by applying Proposition \ref{prop-dual-2} and Remark \ref{prop-dual-3}, we see that the Lipschitz cross-norm $\pi$ is uniform. Let $\sum_{i=1}^n \delta_{(x_i,y_i)}\boxtimes e_i\in X\boxtimes E$. For every $h\in\Lipo(X,X)$ and $T\in\L(E;E)$, we have  
\begin{align*}
\pi\left((h\boxtimes T)\left(\sum_{i=1}^n \delta_{(x_i,y_i)}\boxtimes e_i\right)\right)
&=\pi\left(\sum_{i=1}^n \delta_{(h(x_i),h(y_i))}\boxtimes T(e_i)\right)\\
&\leq\sum_{i=1}^n d(h(x_i),h(y_i))\left\|T(e_i)\right\|\\
&\leq\sum_{i=1}^n\Lip(h)d(x_i,y_i)\left\|T\right\|\left\|e_i\right\|\\
&=\Lip(h)\left\|T\right\|\sum_{i=1}^n d(x_i,y_i)\left\|e_i\right\|.
\end{align*}
The value of $(h\boxtimes T)\left(\sum_{i=1}^n \delta_{(x_i,y_i)}\boxtimes e_i\right)$ is independent of the representation of $\sum_{i=1}^n \delta_{(x_i,y_i)}\boxtimes e_i$ by Lemma \ref{lem-7}, and therefore we conclude that 
$$
\pi\left((h\boxtimes T)\left(\sum_{i=1}^n \delta_{(x_i,y_i)}\boxtimes e_i\right)\right)
\leq\Lip(h)\left\|T\right\|\pi\left(\sum_{i=1}^n \delta_{(x_i,y_i)}\boxtimes e_i\right).
$$
\end{proof}

The Lipschitz projective norm on $X\boxtimes E$ and the dual norm of the norm $\Lip$ of $\Lipo(X,E^*)$ induced on $X\boxtimes E$ coincide as we see next.

\begin{corollary}\label{pi es L}
Let $X$ be a pointed metric space and $E$ a Banach space. Then $\pi=L$ on $X\boxtimes E$.
\end{corollary}

\begin{proof}
By Theorem \ref{teo-pro-norm}, $L\leq\pi$. To prove that $L\geq\pi$, suppose by contradiction that $L(u_0)<1<\pi(u_0)$ for some $u_0\in X\boxtimes E$. Denote $B=\{u\in X\boxtimes E\colon \pi(u)\leq 1\}$. Clearly, $B$ is a closed and convex set in $X\boxtimes_\pi E$. Applying the Hahn--Banach separation theorem to $B$ and $\{u_0\}$, we obtain a functional $\eta\in(X\boxtimes_\pi E)^*$ such that 
$$
1=\|\eta\|=\sup\{\Re\,\eta(u)\colon u\in B \} < \Re\,\eta(u_0).
$$
Define $f\colon X\to E^*$ by $\langle f(x),e\rangle=\eta\left(\delta_{(x,0)}\boxtimes e \right)$ 
for all $e\in E$ and $x\in X$. It is easy to prove that $f$ is well defined and $f\in\Lipo(X,E^*)$ with $\Lip(f)\leq 1$. Moreover
$u(f)=\eta(u)$ for all $u\in X\boxtimes E$. Therefore
$L(u_0)\geq |u_0(f)|\geq \Re\,u_0(f)=\Re\,\eta(u_0)$, so $L(u_0)>1$ and this is a contradiction.
\end{proof}

\begin{theorem}\label{teo-greatest}
$\pi$ is the greatest Lipschitz cross-norm on $X\boxtimes E$.
\end{theorem}

\begin{proof}
We have seen in Theorem \ref{teo-pro-norm} that $\pi$ is a Lipschitz cross-norm on $X\boxtimes E$. Now, let $\alpha$ be a Lipschitz cross-norm on $X\boxtimes E$ and let $u\in X\boxtimes E$. If $\sum_{i=1}^n \delta_{(x_i,y_i)}\boxtimes e_i$ is a representation of $u$, we have 
$$
\alpha(u)
=\alpha\left(\sum_{i=1}^n \delta_{(x_i,y_i)}\boxtimes e_i\right)
\leq\sum_{i=1}^n\alpha\left(\delta_{(x_i,y_i)}\boxtimes e_i\right)
=\sum_{i=1}^n d(x_i,y_i)\left\|e_i\right\|.
$$
Now the very definition of $\pi$ gives $\alpha(u)\leq\pi(u)$.
\end{proof}

Dualizable Lipschitz cross-norms on $X\boxtimes E$ are characterized by being between the Lipschitz injective and Lipschitz projective norms.

\begin{proposition}\label{prop-carac-dualizable}
A norm $\alpha$ on $X\boxtimes E$ is a dualizable Lipschitz cross-norm if and only if $\varepsilon\leq\alpha\leq\pi$. 
\end{proposition}

\begin{proof}
If $\alpha$ is a dualizable Lipschitz cross-norm on $X\boxtimes E$, then $\varepsilon\leq\alpha\leq\pi$ by Theorems \ref{teo-least} and \ref{teo-greatest}. Conversely, if $\alpha$ is a norm on $X\boxtimes E$ that lies between $\varepsilon$ and $\pi$, then $\alpha(\delta_{(x,y)}\boxtimes e)=d(x,y)\left\|e\right\|$ follows immediately from the fact that $\varepsilon$ and $\pi$ are Lipschitz cross-norms. Let $g\in X^\#$ and $\phi\in E^*$. Then 
$$
\left|\sum_{i=1}^n\left(g(x_i)-g(y_i)\right)\left\langle\phi,e_i\right\rangle\right|
\leq\Lip(g)\left\|\phi\right\|\varepsilon\left(\sum_{i=1}^n \delta_{(x_i,y_i)}\boxtimes e_i\right)
\leq\Lip(g)\left\|\phi\right\|\alpha\left(\sum_{i=1}^n \delta_{(x_i,y_i)}\boxtimes e_i\right)
$$
for all $\sum_{i=1}^n \delta_{(x_i,y_i)}\boxtimes e_i\in X\boxtimes E$, and so the Lipschitz cross-norm $\alpha$ is dualizable.
\end{proof}

Let us recall that a completion of a normed space $E$ is a Banach space $\widetilde{E}$ that includes a dense linear subspace isometric to $E$. Every normed space has a completion and the completion is unique up to isometric isomorphism. By \cite[Lemma 3.100]{fhhmz}, every element $\widetilde{e}$ of the completion $\widetilde{E}$ of $E$ can be written as $\widetilde{e}=\sum_{n=1}^\infty e_n$, where $e_n\in E$ and $\sum_{n=1}^\infty\left\|e_n\right\|<\infty$. Moreover, $\left\|\widetilde{e}\right\|=\inf\left\{\sum_{n=1}^\infty\left\|e_n\right\|\right\}$, where the infimum is taken over all series in $E$ summing up to $\widetilde{e}$. 
Combining this with Definition \ref{def-pro-norm}, we obtain the following result.

\begin{theorem}\label{theo-projective Lipschitz tensor product}
Every element $u\in X\widehat{\boxtimes}_\pi E$ admits a representation  
$$
u=\sum_{i=1}^\infty \delta_{(x_i,y_i)}\boxtimes e_i
$$
such that
$$
\sum_{i=1}^\infty d(x_i,y_i)\left\|e_i\right\|<\infty .
$$
Moreover,
$$
\pi(u)=\inf\left\{\sum_{i=1}^\infty d(x_i,y_i)\left\|e_i\right\|\colon u=\sum_{i=1}^\infty\delta_{(x_i,y_i)}\boxtimes e_i,\ \sum_{i=1}^\infty d(x_i,y_i)\left\|e_i\right\|<\infty\right\}.
$$
The dual pairing satisfies the formula
$$
\left\langle\sum_{i=1}^\infty \delta_{(x_i,y_i)}\boxtimes e_i,f\right\rangle=\sum_{i=1}^\infty \left\langle f(x_i)-f(y_i),e_i\right\rangle
$$
for all $f\in\Lipo(X,E^*)$. 
\end{theorem}

\begin{proof}
We follow the proof of \cite[Lemma 3.100]{fhhmz}. 
Obviously, $\overline{B_{X\boxtimes_\pi E}}=B_{X\widehat{\boxtimes}_\pi E}$. 
We may assume, without loss of generality, that $u\in B_{X\widehat{\boxtimes}_\pi E}$. 
Fix $\varepsilon >0$. There exists $u_1\in B_{X\boxtimes_\pi E}$ such that $\pi(u-u_1)\leq \varepsilon$. 
Moreover, by Definition \ref{def-pro-norm}, we can take a representation of $u_1$, 
$\sum_{i=1}^{m_1}\delta_{(x_i,y_i)}\boxtimes e_i$, such that
$$
\sum_{i=1}^{m_1} d(x_i,y_i) \|e_i\| <\pi(u_1)+\varepsilon.
$$
Since $(u-u_1)/\varepsilon \in B_{X\widehat{\boxtimes}_\pi E}$, there is 
$u_2\in \varepsilon B_{X\boxtimes_\pi E}$ such that $\pi((u-u_1)/\varepsilon-u_2/\varepsilon)\leq 1/2$, that is,
$\pi(u-u_1-u_2)\leq \varepsilon/2$. 
As before, by Definition \ref{def-pro-norm}, we can take a representation of $u_2$, 
$\sum_{i=m_1+1}^{m_2}\delta_{(x_i,y_i)}\boxtimes e_i$, such that
$$
\sum_{i=m_1+1}^{m_2} d(x_i,y_i) \|e_i\| <\pi(u_2)+\frac{\varepsilon}{2}.
$$
Find $u_3\in (\varepsilon/2) B_{X\boxtimes_\pi E}$ such that 
$\pi((2/\varepsilon)(u-u_1-u_2)-(2/\varepsilon)u_3)\leq 1/2$, i. e.,
$\pi(u-u_1-u_2-u_3)\leq \varepsilon/2^2$, and take a representation of $u_3$, 
$\sum_{i=m_2+1}^{m_3}\delta_{(x_i,y_i)}\boxtimes e_i$, such that
$$
\sum_{i=m_2+1}^{m_3} d(x_i,y_i) \|e_i\| <\pi(u_3)+\frac{\varepsilon}{2^2}.
$$
Proceed recursively to obtain sequences $\{u_n\}$ in $X\boxtimes_\pi E$, $\{m_n\}$ in $\mathbb{N}$,
$\{x_n\},\{y_n\}$ in $X$ and $\{e_n\}$ in $E$ verifying that $u_n\in \left(\varepsilon/2^{n-2}\right) B_{X\boxtimes_\pi E}$,
$u_n=\sum_{i=m_{n-1}+1}^{m_n} \delta_{(x_i,y_i)}\boxtimes e_i$ and
$$
\pi\left( u-\sum_{j=1}^n u_j\right)\leq \frac{\varepsilon}{2^{n-1}},\qquad
\sum_{i=m_{n-1}+1}^{m_n}  d(x_i,y_i) \|e_i\| <\pi(u_n)+\frac{\varepsilon}{2^{n-1}}
$$
for all $n\in\mathbb{N}$, $n\geq 2$. Clearly, $u=\sum_{n=1}^\infty u_n$. Furthermore, from 
$\pi(u_1)-\pi(u)\leq \varepsilon$, it follows that
$$
\sum_{n=1}^\infty \pi(u_n) \leq \pi(u)+\varepsilon+\varepsilon\sum_{n=2}^\infty\frac{1}{2^{n-2}} =\pi(u)+3\varepsilon.
$$
Note that the sequence $\{m_n\}$ is strictly increasing and
$$
\sum_{i=1}^n d(x_i,y_i)\|e_i\| \leq \sum_{i=1}^{m_n} d(x_i,y_i)\|e_i\| <
\sum_{j=1}^n \pi(u_j)+\sum_{j=1}^n \frac{\varepsilon}{2^{j-1}} \leq 
\pi(u)+3\varepsilon+ 2\varepsilon=\pi(u)+5\varepsilon
$$
for all $n\in\mathbb{N}$. Then the series $\sum_{i\geq 1} \delta_{(x_i,y_i)}\boxtimes e_i$ 
is absolutely convergent. We calculate its sum. Denote, for each $n\in \mathbb{N}$, 
$S_n=\sum_{i=1}^n \delta_{(x_i,y_i)}\boxtimes e_i$. We have that $\{S_{m_n}\}$ is a partial sequence of 
$\{S_n\}$ with $S_{m_n}=\sum_{j=1}^n u_j$ for all $n\in \mathbb{N}$. Since $\{S_{m_n}\}$ converges to
$\sum_{n=1}^\infty u_n=u$, then $\{S_n\}$ converges to $u$, that is, 
$\sum_{i=1}^\infty \delta_{(x_i,y_i)}\boxtimes e_i=u$.
Finally, from the inequality
$$
\sum_{i=1}^\infty d(x_i,y_i)\|e_i\| \leq \pi(u)+5\varepsilon
$$
and the arbitrariness of $\varepsilon$, we obtain that 
$$
\inf\left\{\sum_{i=1}^\infty d(x_i,y_i)\left\|e_i\right\|\colon 
u=\sum_{i=1}^\infty\delta_{(x_i,y_i)}\boxtimes e_i,\ \sum_{i=1}^\infty d(x_i,y_i)\left\|e_i\right\|<\infty\right\}
\leq \pi(u).
$$
The opposite inequality is obvious.

To check the dual pairing formula, consider a representation of $u$,
$\sum_{i=1}^\infty \delta_{(x_i,y_i)}\boxtimes e_i$. Given $f\in\Lipo(X,E^*)$, we must see that
the series $\sum_{i\geq 1} \langle f(x_i)-f(y_i), e_i\rangle$ converges to $u(f)$. Denote, for each
$i\in \mathbb{N}$, $u_i=\sum_{j=1}^i \delta_{(x_j,y_j)}\boxtimes e_j \in X\boxtimes_\pi E$. Since
$\sum_{j=1}^i \langle f(x_j)-f(y_j), e_j\rangle=u_i(f)$, we have that
$$
\left| u(f)-\sum_{j=1}^i \langle f(x_j)-f(y_j), e_j\rangle \right|=
|u(f)-u_i(f)|\leq L(u-u_i) \Lip(f)=   \pi(u-u_i) \Lip(f).
$$
Taking into account that $\{\pi(u-u_i)\}$ converges to $0$, we obtain 
$\sum_{i=1}^\infty \langle f(x_i)-f(y_i), e_i\rangle= u(f)$.
\end{proof}

Next we consider the boundedness of the linear operator $h\boxtimes T\colon X\boxtimes E\to Y\boxtimes F$ for the Lipschitz projective norms.

\begin{proposition}
Let $X,Y$ be pointed metric spaces and $E,F$ Banach spaces. Let $h\in\Lipo(X,Y)$ and $T\in\L(E;F)$. Then there exists a unique bounded linear operator $h\boxtimes_\pi T\colon X\widehat{\boxtimes}_\pi E\to Y\widehat{\boxtimes}_\pi F$ such that $(h\boxtimes_\pi T)(u)=(h\boxtimes T)(u)$ for all $u\in X\boxtimes E$. Furthermore, $\left\|h\boxtimes_\pi T\right\|=\Lip(h)\left\|T\right\|$.
\end{proposition}

\begin{proof}
Let $u\in X\boxtimes E$ and let $\sum_{i=1}^n \delta_{(x_i,y_i)}\boxtimes e_i$ be a representation of $u$. We have 
\begin{align*}
\pi((h\boxtimes T)(u))
&=\pi\left(\sum_{i=1}^n \delta_{(h(x_i),h(y_i))}\boxtimes T(e_i)\right)\\
&\leq\sum_{i=1}^n d(h(x_i),h(y_i))\left\|T(e_i)\right\|\\
&\leq\Lip(h)\left\|T\right\|\sum_{i=1}^n d(x_i,y_i)\left\|e_i\right\|.
\end{align*}
An appeal to Definition \ref{def-pro-norm} yields $\pi((h\boxtimes T)(u))\leq\Lip(h)\left\|T\right\|\pi(u)$. Therefore $h\boxtimes T$ is bounded from $X\boxtimes_\pi E$ to $Y\boxtimes_\pi F$ and $\left\|h\boxtimes T\right\|\leq\Lip(h)\left\|T\right\|$. Moreover, the converse estimate follows easily from
$$
d(h(x),h(y))\left\|T(e)\right\|
=\pi(\delta_{(h(x),h(y))}\boxtimes T(e))
=\pi((h\boxtimes T)(\delta_{(x,y)}\boxtimes e))
\leq \left\|h\boxtimes T\right\|d(x,y)\left\|e\right\|
$$
for all $x,y\in X$ and $e\in E$. Thus, we have $\left\|h\boxtimes T\right\|=\Lip(h)\left\|T\right\|$. Finally, it is well known that $h\boxtimes T$ has a unique bounded linear extension to an operator $h\boxtimes_\pi T\colon X\widehat{\boxtimes}_\pi E\to Y\widehat{\boxtimes}_\pi F$ with $\left\|h\boxtimes_\pi T\right\|=\Lip(h)\left\|T\right\|$.
\end{proof}

It turns out that there is a very close relationship between the Lipschitz projective norm and the projective tensor product of Banach spaces.
In fact, just as it was the case for the injective norm in Proposition \ref{prop-Lipschitz-injective-is-Banach-injective}, the Lipschitz projective norm on $X \boxtimes E$ can be identified with the projective norm on the tensor product of $\F(X)$ and $E$.
The authors wish to thank Richard Haydon for suggesting that this might be true.

\begin{proposition}\label{prop-Lipschitz-projective-is-Banach-projective}
The map $I \colon  X\boxtimes_{\pi} E \to \F(X) \otimes_\pi E$, defined by 
$$
I(u) = \sum_{i=1}^n(\delta_{x_i}-\delta_{y_i}) \otimes e_i
$$
for $u=\sum_{i=1}^n \delta_{(x_i,y_i)}\boxtimes e_i\in X\boxtimes E$, is a linear isometry. Moreover,  $X\widehat{\boxtimes}_{\pi} E$ is isometrically isomorphic to $\F(X) \widehat{\otimes}_\pi E$.
\end{proposition}

\begin{proof}
As in the proof of Proposition \ref{prop-Lipschitz-injective-is-Banach-injective}, Proposition \ref{pro-0} guarantees that the map $I$ is well defined (and it is clearly linear).
Letting $u=\sum_{i=1}^n \delta_{(x_i,y_i)}\boxtimes e_i\in X\boxtimes E$, note that using the fact that the map $x \mapsto \delta_x$ is an isometry from $X$ into $\F(X)$,
$$
\n{ \sum_{i=1}^n(\delta_{x_i}-\delta_{y_i}) \otimes e_i }_{\F(X) \otimes_\pi E}
\le \sum_{i=1}^n \n{\delta_{x_i}-\delta_{y_i}}_{\F(X)} \n{e_i}
= \sum_{i=1}^n d(x_i,y_i) \n{e_i}.
$$
Taking the infimum over all representations of $u$, we conclude that $\n{I(u)} \le \pi(u)$.
Let $\eta>0$ be given. From Corollary \ref{pi es L} there exists $f\in\Lipo(X,E^*)$ with $\Lip(f)\leq 1$ such that $\left|\left\langle u,f\right\rangle\right| > \pi(u)-\eta$, where the pairing is the one given in Theorem \ref{theo-projective Lipschitz tensor product}.
By Theorem \ref{thm-Arens-Eells}, the linear extension $\hat{f} \colon  \F(X) \to E^*$ of $f$ has norm at most one. From the properties of the projective tensor product of Banach spaces,
the dual of $\F(X) \widehat{\otimes}_\pi E$ can be identified with $\mathcal{L}(\F(X),E^*)$, where the pairing is given by
$$
\Bpair{\sum_{i=1}^n \gamma_i \otimes e_i}{T}=\sum_{i=1}^n \pair{T\gamma_i}{e_i} \qquad \left(\gamma_i \in \F(X),\; e_i \in E,\; T\in\mathcal{L}(\F(X),E^*)\right).
$$
In particular,
$$
\left\|I(u)\right\|\ge\left|\left\langle u,\hat{f}\right\rangle\right|
=\left|\sum_{i=1}^n \pair{\hat{f}(\delta_{x_i}-\delta_{y_i})}{e_i} \right|
=\left|\sum_{i=1}^n \pair{f(x_i)-f(y_i)}{e_i}\right|
=\left|\left\langle u,f\right\rangle\right|> \pi(u)-\eta.
$$
Letting $\eta$ go to 0, we obtain that $\n{I(u)} \ge \pi(u)$, and thus $I$ is a linear isometry. Since the sums of the form $\sum_{i=1}^n \delta_{(x_i,y_i)}\boxtimes e_i$ and $\sum_{i=1}^n(\delta_{x_i}-\delta_{y_i}) \otimes e_i$ are dense respectively in $X\widehat{\boxtimes}_{\pi} E$ and $\F(X) \widehat{\otimes}_\pi E$, $I$ extends to an isometric isomorphism between $X\widehat{\boxtimes}_{\pi} E$ and $\F(X) \widehat{\otimes}_\pi E$.
\end{proof}

The identification obtained in Proposition  \ref{prop-Lipschitz-projective-is-Banach-projective}
invites to point out the following relationship between Lipschitz tensor product operators and 
usual tensor product operators.

\begin{remark}\label{oper_Lipschitz_project_Banach_project}
Let $X,Y$ be pointed metric spaces and $E,F$ be Banach spaces. 
Let $h\in \Lipo(X,Y)$, \mbox{$T\in \mathcal{L}(E,F)$} and $\widetilde{h}\colon \F(X)\to\F(Y)$ be the induced map given 
in theorem 7.5. Then the diagram

\begin{center}
\begin{picture}(140,82)
\put(13,60){$X\widehat{\boxtimes}_\pi E$}
\put(41,62){\vector(1,0){44}}
\put(50,67){$h\boxtimes_\pi T$}
\put(89,60){$Y\widehat{\boxtimes}_\pi F$}
\put(23,52){\vector(0,-1){40}}
\put(3,0){$\F(X)\widehat{\otimes}_\pi E$}
\put(0,30){$I_{X\widehat{\boxtimes}_\pi E}$}
\put(103,30){$I_{Y\widehat{\boxtimes}_\pi F}$}
\put(44,2){\vector(1,0){40}}
\put(87,0){$\F(Y)\widehat{\otimes}_\pi F$}
\put(50,7){$\widetilde{h}\otimes_\pi T$}
\put(100,52){\vector(0,-1){40}}
\end{picture}
\end{center}
\medskip
conmutes, that is, 
$I_{Y\widehat{\boxtimes}_\pi F}\circ (h\boxtimes_\pi T)=
\left(\widetilde{h}\otimes_\pi T\right)\circ I_{X\widehat{\boxtimes}_\pi E}$.  
\end{remark}

Proposition \ref{prop-Lipschitz-projective-is-Banach-projective} implies in particular that given a pointed metric space $X$, there is a Banach space $A$ such that $X \widehat{\boxtimes}_\pi E$ is isometric to $A\widehat{\otimes}_\pi E$ for every Banach space $E$.
The authors would like to thank Jes\'{u}s Castillo for pointing out a result in categorical Banach space theory that shows this was to be expected. Without going into all the details, let us outline the argument.
First, a theorem of Fuks \cite[Section 6]{Fuks} (a nice presentation can be found in \cite[Proposition 5.6]{Castillo-hitchhiker}, where the reader can also find the definitions of the categorical terms we use below) states the following:
if $\mathfrak{F}$, $\mathfrak{G}$ are two covariant Banach functors such that for any Banach spaces $E$ and $F$, we have that
$\mathcal{L}(\mathfrak{F}(E),F)$ is linearly isometric to $\mathcal{L}(E,\mathfrak{G}(F))$,
then there exists a Banach space $A$ such that for every Banach space $E$,
$\mathfrak{F}(E)$ is linearly isometric to $A \widehat{\otimes}_\pi E$ and $\mathfrak{G}(F)$ is linearly isometric to $\mathcal{L}(A,F)$.
Now consider a fixed pointed metric space $X$. Note that it induces two covariant Banach functors $X \widehat{\boxtimes}_\pi (\cdot)$ and $\Lipo(X,\cdot)$.
Arguments closely related to those that led us to prove Corollary \ref{pi es L} show that, for any Banach spaces $E$ and $F$, we have
$\mathcal{L}(X\widehat{\boxtimes}_\pi E, F)$ is linearly isometric to $\mathcal{L}(E,\Lipo(X,F))$,
so Fuks' result applies.

The space $X\widehat{\boxtimes}_\pi E$ is called the projective Lipschitz tensor product of $X$ and $E$. This term derives from the following result.
Before stating it, recall that a Lipschitz map $f \colon  X \to Z$ is called $C$-co-Lipschitz if for every $x \in X$ and $r>0$, $f(B(x,r))\supset B(f(x),r/C)$. Moreover, it is called a Lipschitz quotient if it is surjective, Lipschitz and co-Lipschitz \cite{bjlps}. 

\begin{theorem}
Let $X, Z$ be pointed metric spaces and $q\colon  X \to Z$ a Lipschitz quotient that is  
$1$-Lipschitz  and $C$-co-Lipschitz for every $C>1$. Let $E,F$ be a Banach spaces and $Q\colon E\to  F$ a quotient operator. Then $q \boxtimes_\pi Q\colon X\widehat{\boxtimes}_\pi E\to Z\widehat{\boxtimes}_\pi F$ is a quotient operator.
\end{theorem}

\begin{proof}
	Thanks to Proposition \ref{prop-Lipschitz-projective-is-Banach-projective}, Remark \ref{oper_Lipschitz_project_Banach_project}
	and the behavior of the projective tensor norm with respect to quotients \cite[Proposition 2.5]{Ryan},
	it suffices to prove that if $q \colon  X \to Z$ is such a Lipschitz quotient then the induced map $\tilde{q} \colon  \F(X) \to \F(Z)$ is a linear quotient operator.
	Notice that from Theorem \ref{thm-Arens-Eells}, $\| \tilde{q} \| = \Lip(q) = 1$.
	Now let $u \in \F(Z)$, and let $\varepsilon>0$.
	From Proposition \ref{prop-Lipschitz-projective-is-Banach-projective},
	$\F(X) \equiv \F(X) \widehat{\otimes}_\pi \K \equiv X \widehat{\boxtimes}_\pi \K$.
	Thus from Theorem \ref{theo-projective Lipschitz tensor product}, there exists a representation $u = \sum_{j=1}^\infty \delta_{(z_j,z'_j)} \boxtimes a_j$ such that
	$(1+\varepsilon) \pi(u) \ge \sum_{j=1}^\infty |a_j|d(z_j,z'_j)$.
	For each $j$, choose $x_j,x'_j \in X$ such that $q(x_j) = z_j$, $q(x'_j) = z'_j$ and $d(x_j,x'_j) \le (1+\varepsilon)d(z_j,z'_j)$. Setting $u' = \sum_{j=1}^\infty \delta_{(x_j,x'_j)} \boxtimes a_j$, clearly $\tilde{q}(u') = u$ (so it follows that $\tilde{q}$ is surjective) and
	$$
	\pi(u') \le \sum_{j=1}^\infty |a_j|d(x_j,x'_j) \le (1+\varepsilon)\sum_{j=1}^\infty |a_j|d(z_j,z'_j) \le (1+\varepsilon)^2 \pi(u).
	$$
	Since this holds for all $\varepsilon>0$, it follows that $\pi(u)=\inf\left\{\pi(u')\colon\tilde{q}(u')=u\right\}$.
\end{proof}

In a similar manner, the projective norm respects complemented subspaces.
We say that a subset $Z \subset X$ that contains the point 0 is a Lipschitz retract of $X$, or that it is Lipschitz complemented in $X$, if there exists a Lipschitz map (called a Lipschitz retraction) $r \colon  X \to Z$ such that $r(z) = z$ for all $z \in Z$.

\begin{proposition}\label{prop:projective-respects-complemented-subspaces}
Let $Z$ be a Lipschitz retract of $X$, and let $F$ be a complemented subspace of $E$.
Then $Z \widehat{\boxtimes}_\pi F$ is complemented in $X \widehat{\boxtimes}_\pi E$ and the norm on $Z \widehat{\boxtimes}_\pi F$ induced by the Lipschitz projective norm of $X \widehat{\boxtimes}_\pi E$ is equivalent to the Lipschitz projective norm on $Z \widehat{\boxtimes}_\pi F$.
If $Z$ is Lipschitz complemented with a Lipschitz retraction of Lipschitz constant one and $F$ is complemented by a linear projection of norm one, then $Z \widehat{\boxtimes}_\pi F$ is a subspace of $X \widehat{\boxtimes}_\pi E$ and is also complemented by a projection of norm one.	
\end{proposition}

\begin{proof}
This follows from the corresponding result for the projective tensor product (see \cite[Proposition 2.4]{Ryan}), after noting that a Lipschitz retraction $r \colon  X \to Z$ (that in particular sends 0 to 0) extends to a linear projection $\tilde{r} \colon  \F(X) \to \F(Z) \subset \F(X)$ with $\n{\tilde{r}} = \Lip(r)$.
\end{proof}

Calculating the projective norm of an element in a tensor product of Banach spaces is generally difficult, but there is a particular case where the calculation is relatively easy: for any Banach space $E$, $\ell_1 \widehat{\otimes}_\pi E$ is isometrically isomorphic to $\ell_1(E)$ (see \cite[Example 2.6]{Ryan}).
In the nonlinear setting, trees play a role analogous to that of $\ell_1$ in the linear theory, so the following result is not surprising.

\begin{proposition}\label{proposition-projective-of-tree}
Let $(X,\mathcal{E})$ be a graph with finite vertex set $X$ and edge set $\mathcal{E}$ which is a tree, that is, it is connected 
and contains no cycles.
Consider $X$ as a pointed metric space, with distance function given by the shortest-path distance and a distinguished fixed point $0 \in X$. 
Let $E$ be a Banach space.
Then $X\boxtimes_\pi E$ is isometrically isomorphic to $\ell_1(\mathcal{E};E)$.
\end{proposition}

\begin{proof}
We say that a vertex $x \in X$ is positive (negative) if it is at an even (respectively, odd) distance from $0 \in X$. Note that, since $(X,\mathcal{E})$ is a tree, the endpoints of every edge in $\mathcal{E}$ have different parities. Therefore every edge $\{x,y\}$ in $\mathcal{E}$ will be written as $(x,y)$ with $x$ negative and $y$ positive.

Consider $x,y \in X$. Let $n=d(x,y)$ and $\{x=z_0,z_1,\dots,z_n=y\}$ be the unique minimal-length path in $(X,\mathcal{E})$
joining $x$ and $y$.
Since, for each $v \in E$,
$$
\n{v}d(x,y) = \sum_{i=1}^n \n{v}d(z_i,z_{i-1}),
$$
in order to calculate $\pi(u)$ for $u \in X \boxtimes E$, it suffices to consider only representations involving $\delta_{(x_i,y_i)}$ with $(x_i,y_i)\in \mathcal{E}$. By the triangle inequality, in the representation we can consolidate all terms corresponding to the same edge $(x_i,y_i)\in \mathcal{E}$, so we can consider only representations of the form
$$
u = \sum_{(x,y)\in \mathcal{E}} \delta_{(x,y)} \boxtimes v_{(x,y)}.
$$
But, for each $u \in X \boxtimes E$, there is only one such representation. 
To see that we use the following claim, which can be proved by induction on the size of the tree:
given $(x_0,y_0)\in \mathcal{E}$, there exists a function $g\in X^\#$ such that $g(x_0)-g(y_0) \neq 0$ and
$g(x)-g(y)=0$ for all $(x,y)\in \mathcal{E}$ with $(x,y)\neq (x_0,y_0)$.

Now let $\sum_{(x,y)\in \mathcal{E}} \delta_{(x,y)} \boxtimes e_{(x,y)}$ and 
$\sum_{(x,y)\in \mathcal{E}} \delta_{(x,y)} \boxtimes v_{(x,y)}$
be two representations of $u\in X\boxtimes E$. Given $(x_0,y_0)\in\mathcal{E}$, take the function
$g\in X^\#$ of the previous claim. Then, by Proposition \ref{pro-0}, we have that 
$$
0=\sum_{(x,y)\in \mathcal{E}} (g(x)-g(y))(e_{(x,y)}-v_{(x,y)})=(g(x_0)-g(y_0))(e_{(x_0,y_0)}-v_{(x_0,y_0)}),
$$
and thus $e_{(x_0,y_0)}=v_{(x_0,y_0)}$. The arbitrariness of $(x_0,y_0)$ shows that the representation of $u$ is unique. If we define $J \colon  X \boxtimes_\pi E \to \ell_1(\mathcal{E};E)$ by
$u \mapsto (v_{(x,y)})_{(x,y) \in \mathcal{E}}$, $J$ is then clearly an isometric isomorphism between $X\boxtimes_\pi E$ and $\ell_1(\mathcal{E};E)$.
\end{proof}

More generally, in the linear case we have that $L_1(\mu)\widehat{\otimes}_\pi E$ is isometrically isomorphic to $L_1(\mu; E)$ for any measure $\mu$ (see \cite[Example 2.19]{Ryan}).
In our nonlinear setting, a possible analogue will be given by a generalization of Proposition \ref{proposition-projective-of-tree} to a more general class of metric trees.
This will depend heavily on the identification of the Lipschitz-free space over such trees carried out in \cite{Godard}.
Before stating the result, let us recall a definition.
An $\R$-tree is a metric space $X$ satisfying the following two conditions: (1) For any points $a$ and $b$ in $X$, there exists a unique isometry $\phi$ of the closed interval $[0,d(a,b)]$ into $X$ such that $\phi(0)=a$ and $\phi(d(a,b)) = b$; (2) Any one-to-one continuous mapping $\varphi \colon  [0,1] \to X$ has the same range as the isometry $\phi$ associated to the points $a = \varphi(0)$ and $b=\varphi(1)$.

\begin{corollary}\label{cor:vector-valued-godard}
Let $X$ be an $\R$-tree and $E$ a Banach space. Then there exists a measure $\mu$ such that $X \widehat{\boxtimes}_\pi E$ is isometric to $L_1(\mu; E)$.
\end{corollary}

\begin{proof}
By \cite[Corollary 3.3]{Godard}, there exists a measure $\mu$ such that $\F(X)$ is isometrically isomorphic to $L_1(\mu)$.
From Proposition \ref{prop-Lipschitz-projective-is-Banach-projective}, $X \widehat{\boxtimes}_\pi E$ is isometrically isomorphic to $\F(X) \widehat{\otimes}_\pi E$.
Finally, from \cite[Example 2.19]{Ryan}, $L_1(\mu) \widehat{\otimes}_\pi E$ is isometric to $L_1(\mu; E)$.
\end{proof}

We finish this section obtaining a universal property for bounded linear operators which follows from the universal property of the projective tensor product of Banach spaces, since $X \boxtimes_\pi E \equiv \F(X) \otimes_\pi E$ by Proposition \ref{prop-Lipschitz-projective-is-Banach-projective}.

\begin{proposition}
Let $X$ be a pointed metric space and $E$ a Banach space.
Then the map $(x,e)\mapsto \delta_{(x,0)}\boxtimes e$, from $X\times E$ into $X\boxtimes_\pi E$, satisfies:
\begin{enumerate}
\item[a)] For each $e\in E$, the function $x\mapsto \delta_{(x,0)}\boxtimes e$, from $X$ into $X\boxtimes_\pi E$, belongs to $\Lipo(X,X\boxtimes_\pi E)$.
\item[b)] Given $x\in X$, the map $e\mapsto \delta_{(x,0)}\boxtimes e$, from $E$ into  $X\boxtimes_\pi E$, is a bounded linear operator.
\end{enumerate}
Moreover, for each normed space $F$ and for each map $\psi\colon X\times E \to F$ verifying a) and b) (that is, $\psi$ is a Lipschitz operator in the first variable and a bounded linear operator in the second one), there is a unique bounded
linear map $\widetilde{\psi}\colon X\boxtimes_\pi E \to F$ such that 
$\widetilde{\psi}\left(\delta_{(x,0)}\boxtimes e \right)=\psi(x,e)$ for all $x\in X$ and $e\in E$. 
$$
\begin{picture}(100,72)
\put(0,58){$X\times E$}
\put(28,60){\vector(1,0){40}}
\put(45,65){$\psi$}
\put(72,58){$F$}
\put(10,50){\vector(0,-1){40}}
\put(0,0){$X\boxtimes_\pi E$}
\multiput(27,13)(7,7){5}{\rotatebox{45}{\line(1,0){5}}}
\put(55,48){\rotatebox{45}{\vector(1,0){8}}}
\put(50,22){$\widetilde{\psi}$}
\end{picture}
$$
\end{proposition} 


\section{The Lipschitz p-nuclear norms}\label{7}

We introduce now the Lipschitz $p$-nuclear norms $d_p$ on $X\boxtimes E$ for $1\leq p\leq\infty$. 
They are Lipschitz versions of the known tensor norms of Chevet \cite{c} and Saphar \cite{s}. 
Similar versions were introduced in \cite{ch11} on spaces of $E$-valued molecules on $X$, 
where they were shown to be in duality with spaces of Lipschitz $p'$-summing maps.

\begin{definition}\label{def-che-norms}
Let $1\leq p\leq\infty$. Let $E$ be a Banach space and let $e_1,\ldots,e_n\in E$. Define: 
$$
\left\|(e_1,\ldots,e_n)\right\|_p
=\left\{\begin{array}{lll}
\displaystyle{\left(\sum_{i=1}^n\left\|e_i\right\|^p\right)^{\frac{1}{p}}}&\text{ if }& 1\leq p<\infty ,\\
& & \\
\max_{1\leq i\leq n}\left\|e_i\right\|&\text{ if }& p=\infty .
\end{array}\right.
$$
Let $X$ be a pointed metric space, $x_1,\ldots,x_n,y_1,\ldots,y_n\in X$ and 
$\lambda_1,\ldots,\lambda_n\in\R^+$. Define: 
$$
\left\|(\lambda_1\delta_{(x_1,y_1)},\ldots,\lambda_n\delta_{(x_n,y_n)})\right\|_p^{Lw}
=\left\{\begin{array}{lll}
\sup_{g\in B_{X^\#}}\displaystyle{\left(\sum_{i=1}^n \lambda_i^p\left|g(x_i)-g(y_i)\right|^p\right)^{\frac{1}{p}}}&\text{ if }& 1\leq p<\infty ,\\
& & \\
\sup_{g\in B_{X^\#}}\displaystyle{\left(\max_{1\leq i\leq n}\lambda_i\left|g(x_i)-g(y_i)\right|\right)} &\text{ if }& p=\infty .
\end{array}\right.
$$
Let $p'$ be the conjugate index of $p$ defined by $p'=p/(p-1)$ if $p\neq 1$, $p'=\infty$ if $p=1$, and $p'=1$ if $p=\infty$. 
 
For each $u\in X\boxtimes E$, define:
$$
d_p(u)=\inf\left\{
\left\|(\lambda_1\delta_{(x_1,y_1)},\ldots,\lambda_n\delta_{(x_n,y_n)})\right\|_{p'}^{Lw}\left\|(e_1,\ldots,e_n)\right\|_p
\colon u=\sum_{i=1}^n \lambda_i\delta_{(x_i,y_i)}\boxtimes e_i\right\},
$$
the infimum being taken over all representations of $u$.
\end{definition}

\begin{theorem}\label{teo-che-norms}
For $1\leq p\leq \infty$, $d_p$ is a uniform and dualizable Lipschitz cross-norm on $X\boxtimes E$.
\end{theorem}

\begin{proof}
Let $u\in X\boxtimes E$ and let $\sum_{i=1}^n\lambda_i\delta_{(x_i,y_i)}\boxtimes e_i$ be a representation of $u$. 
Clearly, $d_p(u)\geq 0$. Let $\lambda\in\K$. 
Since $\sum_{i=1}^n \lambda_i\delta_{(x_i,y_i)}\boxtimes \lambda e_i$ is a representation of 
$\lambda u$, Definition \ref{def-che-norms} yields
$$
d_p(\lambda u)
\leq\left\|(\lambda_1\delta_{(x_1,y_1)},\ldots,\lambda_n\delta_{(x_n,y_n)})\right\|_{p'}^{Lw}\left\|(\lambda e_1,\ldots,\lambda e_n)\right\|_p
=\left|\lambda\right|\left\|(\lambda_1\delta_{(x_1,y_1)},\ldots,\lambda_n\delta_{(x_n,y_n)})\right\|_{p'}^{Lw}\left\|(e_1,\ldots,e_n)\right\|_p.
$$
If $\lambda=0$, it follows that $d_p(\lambda u)=0=\left|\lambda\right|d_p(u)$. For $\lambda\neq 0$, since the preceding inequality holds for every representation of $u$, we deduce that $d_p(\lambda u)\leq\left|\lambda\right|d_p(u)$. For the converse estimate, note that $d_p(u)=d_p(\lambda^{-1}(\lambda u))\leq\left|\lambda^{-1}\right|d_p(\lambda u)$ by using the proved inequality, thus $\left|\lambda\right|d_p(u)\leq d_p(\lambda u)$ and hence $d_p(\lambda u)=\left|\lambda\right|d_p(u)$.

We prove the triangular inequality for $d_p$ as follows. Let $u,v\in X\boxtimes E$ and let $\varepsilon>0$. If $u=0$ or $v=0$, there is nothing to prove. Assume $u\neq 0\neq v$. We can choose representations for $u$ and $v$, say
$$
u=\sum_{i=1}^n \lambda_i\delta_{(x_i,y_i)}\boxtimes e_i,\qquad
v=\sum_{i=1}^m \lambda'_i\delta_{(x'_i,y'_i)}\boxtimes e'_i, 
$$
such that
$$
\left\|(\lambda_1\delta_{(x_1,y_1)},\ldots,\lambda_n\delta_{(x_n,y_n)})\right\|_{p'}^{Lw}\left\|(e_1,\ldots,e_n)\right\|_p\leq d_p(u)+\varepsilon
$$
and 
$$
\left\|(\lambda'_1\delta_{(x'_1,y'_1)},\ldots,\lambda'_m\delta_{(x'_m,y'_m)})\right\|_{p'}^{Lw}\left\|(e'_1,\ldots,e'_m)\right\|_p\leq d_p(v)+\varepsilon .
$$ 
Fix $r,s\in\mathbb{R}^+$ arbitrarily and define 
$$
\lambda''_i\delta_{(x''_i,y''_i)}=\left\{\begin{array}{lll}
r^{-1}\lambda_i\delta_{(x_i,y_i)}&\text{ if } i=1,\ldots,n,\\
s^{-1}\lambda'_{i-n}\delta_{(x'_{i-n},y'_{i-n})}&\text{ if } i=n+1,\ldots,n+m,
\end{array}\right.
\qquad
e''_i=\left\{\begin{array}{lll}
re_i&\text{ if } i=1,\ldots,n,\\
se'_{i-n}&\text{ if } i=n+1,\ldots,n+m.
\end{array}\right.
$$
It is plain that $u+v=\sum_{i=1}^{n+m}\lambda''_i\delta_{(x''_i,y''_i)}\boxtimes e''_i$ and therefore we have
$$
d_p(u+v)\leq\left\|(\lambda''_1\delta_{(x''_1,y''_1)},\ldots,\lambda''_{n+m}\delta_{(x''_{n+m},y''_{n+m})})\right\|_{p'}^{Lw}\left\|(e''_1,\ldots,e''_{n+m})\right\|_p.
$$
We distinguish three cases. 

\begin{cs} $1<p<\infty$. An easy verification yields  
$$
\left(\left\|(\lambda''_1\delta_{(x''_1,y''_1)},\ldots,\lambda''_{n+m}\delta_{(x''_{n+m},y''_{n+m})})\right\|_{p'}^{Lw}\right)^{p'}
\leq\left(r^{-1}\left\|(\lambda_1\delta_{(x_1,y_1)},\ldots,\lambda_n\delta_{(x_n,y_n)})\right\|_{p'}^{Lw}\right)^{p'}
+\left(s^{-1}\left\|(\lambda'_1\delta_{(x'_1,y'_1)},\ldots,\lambda'_m\delta_{(x'_m,y'_m)})\right\|_{p'}^{Lw}\right)^{p'} 
$$
and
$$
\left(\left\|(e''_1,\ldots,e''_{n+m})\right\|_p\right)^p
=\left(r\left\|(e_1,\ldots,e_n)\right\|_p\right)^p+\left(s\left\|(e'_1,\ldots,e'_m)\right\|_p\right)^p.
$$
Using Young's inequality, it follows that 
\begin{align*}
d_p(u+v)
&\leq\left\|(\lambda''_1\delta_{(x''_1,y''_1)},\ldots,\lambda''_{n+m}\delta_{(x''_{n+m},y''_{n+m})})\right\|_{p'}^{Lw}\left\|(e''_1,\ldots,e''_{n+m})\right\|_p \\
&\leq\frac{1}{p'}\left(\left\|(\lambda''_1\delta_{(x''_1,y''_1)},\ldots,\lambda''_{n+m}\delta_{(x''_{n+m},y''_{n+m})})\right\|_{p'}^{Lw}\right)^{p'}+\frac{1}{p}\left(\left\|(e''_1,\ldots,e''_{n+m})\right\|_p\right)^p\\
&\leq\frac{r^{-p'}}{p'}\left(\left\|(\lambda_1\delta_{(x_1,y_1)},\ldots,\lambda_n\delta_{(x_n,y_n)})\right\|_{p'}^{Lw}\right)^{p'}+\frac{r^p}{p}\left(\left\|(e_1,\ldots,e_n)\right\|_p\right)^p\\
&+\frac{s^{-p'}}{p'}\left(\left\|(\lambda'_1\delta_{(x'_1,y'_1)},\ldots,\lambda'_m\delta_{(x'_m,y'_m)})\right\|_{p'}^{Lw}\right)^{p'}+\frac{s^p}{p}\left(\left\|(e'_1,\ldots,e'_m)\right\|_p\right)^p.
\end{align*}
Since $r,s$ were arbitrary in $\R^+$, taking 
\begin{align*}
r&=(d_p(u)+\varepsilon)^{-1/p'}\left\|(\lambda_1\delta_{(x_1,y_1)},\ldots,\lambda_n\delta_{(x_n,y_n)})\right\|_{p'}^{Lw},\\
s&=(d_p(v)+\varepsilon)^{-1/p'}\left\|(\lambda'_1\delta_{(x'_1,y'_1)},\ldots,\lambda'_m\delta_{(x'_m,y'_m)})\right\|_{p'}^{Lw},
\end{align*}
we obtain that $d_p(u+v)\leq d_p(u)+d_p(v)+2\varepsilon$.
\end{cs}

\begin{cs} $p=1$. 
Now we have  
\begin{align*}
d_1(u+v)
&\leq\left\|(\lambda''_1\delta_{(x''_1,y''_1)},\ldots,\lambda''_{n+m}\delta_{(x''_{n+m},y''_{n+m})})\right\|_{\infty}^{Lw}
\left\|(e''_1,\ldots,e''_{n+m})\right\|_1 \\
&=\left(\max\left\{r^{-1}\left\|(\lambda_1\delta_{(x_1,y_1)},\ldots,\lambda_n\delta_{(x_n,y_n)})\right\|_{\infty}^{Lw},
s^{-1}\left\|(\lambda'_1\delta_{(x'_1,y'_1)},\ldots,\lambda'_m\delta_{(x'_m,y'_m)})\right\|_{\infty}^{Lw}\right\}\right)\\
&\cdot \left(r\left\|(e_1,\ldots,e_n)\right\|_1+s\left\|(e'_1,\ldots,e'_m)\right\|_1\right),
\end{align*}
and taking, in particular, $r=\left\|(\lambda_1\delta_{(x_1,y_1)},\ldots,\lambda_n\delta_{(x_n,y_n)})\right\|_{\infty}^{Lw}$ and $s=\left\|(\lambda'_1\delta_{(x'_1,y'_1)},\ldots,\lambda'_m\delta_{(x'_m,y'_m)})\right\|_{\infty}^{Lw}$, we infer that $d_1(u+v)\leq d_1(u)+d_1(v)+2\varepsilon$.
\end{cs}

\begin{cs} $p=\infty$. We have 
\begin{align*}
d_\infty(u+v)
&\leq\left\|(\lambda''_1\delta_{(x''_1,y''_1)},\ldots,\lambda''_{n+m}\delta_{(x''_{n+m},y''_{n+m})})\right\|_1^{Lw}
\left\|(e''_1,\ldots,e''_{n+m})\right\|_\infty \\
&\leq\left(r^{-1}\left\|(\lambda_1\delta_{(x_1,y_1)},\ldots,\lambda_n\delta_{(x_n,y_n)})\right\|_1^{Lw}
+s^{-1}\left\|(\lambda'_1\delta_{(x'_1,y'_1)},\ldots,\lambda'_n\delta_{(x'_n,y'_n)})\right\|_1^{Lw}\right)\\
&\cdot\left(\max\left\{r\left\|(e_1,\ldots,e_n)\right\|_\infty,s\left\|(e'_1,\ldots,e'_m)\right\|_\infty\right\}\right),
\end{align*}
and for $r=\left\|(e_1,\ldots,e_n)\right\|^{-1}_\infty$ and $s=\left\|(e'_1,\ldots,e'_m)\right\|^{-1}_\infty$, we deduce that $d_\infty(u+v)\leq d_\infty(u)+d_\infty(v)+2\varepsilon$.
\end{cs}

In any case, $d_p(u+v)\leq d_p(u)+d_p(v)+2\varepsilon$ and so $d_p(u+v)\leq d_p(u)+d_p(v)$  by the arbitrariness of $\varepsilon$. Hence $d_p$ is a seminorm for $1\leq p\leq \infty$. 
Now, we claim that $\varepsilon(u)\leq d_p(u)\leq\pi(u)$. Indeed, we have
$$
\left|\sum_{i=1}^n \lambda_i\left(g(x_i)-g(y_i)\right)\left\langle\phi,e_i\right\rangle\right|
\leq\sum_{i=1}^n \lambda_i\left|g(x_i)-g(y_i)\right|\left\|e_i\right\|
\leq\left\|(\lambda_1\delta_{(x_1,y_1)},\ldots,\lambda_n\delta_{(x_n,y_n)})\right\|_{p'}^{Lw}\left\|(e_1,\ldots,e_n)\right\|_p
$$
for every $g\in B_{X^\#}$ and $\phi\in B_{E^*}$, where we have used H\"{o}lder's inequality in the case $1<p<\infty$. Therefore $\varepsilon(u)\leq\left\|(\lambda_1\delta_{(x_1,y_1)},\ldots,\lambda_n\delta_{(x_n,y_n)})\right\|_{p'}^{Lw}\left\|(e_1,\ldots,e_n)\right\|_p$, and since it holds for each representation of $u$, we deduce that $\varepsilon(u)\leq d_p(u)$. Since $\varepsilon$ is a Lipschitz cross-norm, this implies that $d_p$ is a norm and that 
$$
\left\|e\right\|d(x,y)=\varepsilon(\delta_{(x,y)}\boxtimes e)\leq d_p(\delta_{(x,y)}\boxtimes e)
$$
for all $x,y\in X$ and $e\in E$. Moreover, Definition \ref{def-che-norms} gives 
$$
d_p(\delta_{(x,y)}\boxtimes e)\leq\left\|e\right\|\sup_{g\in B_{X^\#}}\left|g(x)-g(y)\right|=\left\|e\right\|d(x,y).
$$
Hence $d_p$ is a Lipschitz cross-norm. Then $d_p\leq\pi$ by Theorem \ref{teo-greatest} as we wanted. Now our claim implies that $d_p$ is dualizable by Proposition \ref{prop-carac-dualizable}. 

Finally, to prove that $d_p$ is uniform, take $h\in\Lipo(X,X)$ and $T\in\L(E;E)$. Let $u\in X\boxtimes E$ and pick a representation $\sum_{i=1}^n \lambda_i\delta_{(x_i,y_i)}\boxtimes e_i$ for $u$. We have 
$$
\left\|(\lambda_1\delta_{(h(x_1),h(y_1))},\ldots,\lambda_n\delta_{(h(x_n),h(y_n))})\right\|_{p'}^{Lw}\left\|(T(e_1),\ldots,T(e_n))\right\|_p
\leq\Lip(h)\left\|(\lambda_1\delta_{(x_1,y_1)},\ldots,\lambda_n\delta_{(x_n,y_n)})\right\|_{p'}^{Lw}\left\|T\right\|\left\|(e_1,\ldots,e_n)\right\|_p.
$$
Taking infimum over all the representations of $u$, it follows that $d_p((h\boxtimes T)(u))\leq\Lip(h)\left\|T\right\|d_p(u)$. 
\end{proof}

Next we show that the Lipschitz $1$-nuclear norm $d_1$ is justly the Lipschitz projective norm $\pi$.

\begin{proposition}
For every $u\in X\boxtimes E$, 
$$
d_1(u)=\inf\left\{\sum_{i=1}^n d(x_i,y_i)\left\|e_i\right\|\colon u=\sum_{i=1}^n \delta_{(x_i,y_i)}\boxtimes e_i\right\}
$$
taking the infimum over all representations of $u$.
\end{proposition}

\begin{proof}
Let $u\in X\boxtimes E$ and let $\sum_{i=1}^n \lambda_i\delta_{(x_i,y_i)}\boxtimes e_i$ be a representation of $u$. We have 
\begin{align*}
\pi(u)
&\leq\sum_{i=1}^n \lambda_i d(x_i,y_i)\left\|e_i\right\|\\
&=\sum_{i=1}^n \lambda_i\left(\sup_{g\in B_{X^\#}}\left|g(x_i)-g(y_i)\right|\right)\left\|e_i\right\|\\
&\leq\sum_{i=1}^n\max_{1\leq i\leq n}\left(\lambda_i\sup_{g\in B_{X^\#}}\left|g(x_i)-g(y_i)\right|\right)\left\|e_i\right\|\\
&=\left\|(\lambda_1\delta_{(x_1,y_1)},\ldots,\lambda_n\delta_{(x_n,y_n)})\right\|_\infty^{Lw}\left\|(e_1,\ldots,e_n)\right\|_1
\end{align*}
and therefore $\pi(u)\leq d_1(u)$. The converse inequality follows from Theorems \ref{teo-che-norms} and \ref{teo-greatest}.
\end{proof}

\section{Lipschitz approximable operators}\label{8}

The notions of Lipschitz compact operators and Lipschitz approximable operators from $X$ to $E$ were introduced in \cite{jsv}. Let us recall that a Lipschitz operator $f\in\Lipo(X,E)$ is said to be Lipschitz compact if its Lipschitz image $\left\{(f(x)-f(y))/d(x,y)\colon x,y\in X, \; x\neq y\right\}$ is relatively compact in $E$ and $f$ is said to be Lipschitz approximable if it is the limit in the Lipschitz norm $\Lip$ of a sequence of Lipschitz finite-rank operators from $X$ to $E$. 

We show that the spaces of Lipschitz finite-rank operators and Lipschitz approximable operators can be identified as spaces of continuous linear functionals.

\begin{theorem}\label{teo-identification3}
The map $K\colon X^\#\boxast_{\pi'} E^*\to\LipoF(X,E^*)$, defined by 
$$
K\left(\sum_{j=1}^m g_j\boxtimes\phi_j\right)=\sum_{j=1}^m g_j\cdot\phi_j,
$$ 
is an isometric isomorphism. As a consequence, the space of all Lipschitz approximable operators from $X$ to $E^*$ is isometrically isomorphic to $X^\#\widehat{\boxast}_{\pi'} E^*$.
\end{theorem}

\begin{proof}
By Theorem \ref{theo-identification2}, $K$ is a linear bijection. For any $\sum_{j=1}^m g_j\boxtimes\phi_j\in X^\#\boxast E^*$, we have  
\begin{align*}
\pi'\left(\sum_{j=1}^m g_j\boxtimes\phi_j\right)
&=\sup\left\{\left|\left(\sum_{j=1}^m g_j\boxtimes\phi_j\right)\left(\sum_{i=1}^n \delta_{(x_i,y_i)}\boxtimes e_i\right)\right|
\colon \pi\left(\sum_{i=1}^n \delta_{(x_i,y_i)}\boxtimes e_i\right)\leq 1\right\}\\
&=\sup\left\{\left|\sum_{i=1}^n\left\langle \left(\sum_{j=1}^m g_j\cdot\phi_j\right)(x_i)-\left(\sum_{j=1}^m g_j\cdot\phi_j\right)(y_i),e_i \right\rangle\right|
\colon \pi\left(\sum_{i=1}^n \delta_{(x_i,y_i)}\boxtimes e_i\right)\leq 1\right\}\\
&=\sup\left\{\left|\Lambda\left(\sum_{j=1}^m g_j\cdot\phi_j\right)\left(\sum_{i=1}^n \delta_{(x_i,y_i)}\boxtimes e_i\right)\right|
\colon L\left(\sum_{i=1}^n \delta_{(x_i,y_i)}\boxtimes e_i\right)\leq 1\right\}\\
&=\left\|\Lambda\left(\sum_{j=1}^m g_j\cdot\phi_j\right)\right\|\\
&=\Lip\left(\sum_{j=1}^m g_j\cdot\phi_j\right).
\end{align*}
by using Corollary \ref{pro-dual-norm}, Lemmas \ref{lem-4.1} and \ref{lem-2}, Corollary \ref{pi es L} and Theorem \ref{theo-dual-classic}. Hence $K$ is an isometry. The consequence follows from a known result of Functional Analysis. 
\end{proof}

From Theorems \ref{theo-dual-classic} and \ref{teo-identification3} and Corollary \ref{pi es L}, we infer the next consequence.

\begin{corollary}\label{cor-nuevo5}
The space $X^\#\widehat{\boxast}_{\pi'}E^*$ is isometrically isomorphic to $(X\widehat{\boxtimes}_\pi E)^*$ if and only if $\Lipo(X,E^*)$ is isometrically isomorphic to the space of Lipschitz approximable operators from $X$ to $E^*$. 
\end{corollary}

We recall that a Banach space $E$ is said to have the approximation property if given a compact set $K\subset E$ and $\varepsilon>0$, there is a finite-rank bounded linear operator $T\colon E\to E$ such that $\left\|Tx-x\right\|<\varepsilon$ for every $x\in K$. The approximation property was thoroughly studied by Grothendieck in \cite{g}. In \cite[Corollary 2.5]{jsv}, it was shown that $X^\#$ has the approximation property if and only if the space of all Lipschitz approximable operators from $X$ to $E$ is the space of all Lipschitz compact operators from $X$ to $E$. Using this fact and Theorem \ref{teo-identification3}, we derive the following result.

\begin{corollary}\label{cor-nuevo2}
Let $X$ be a pointed metric space such that $X^\#$ has the approximation property. Then, for any Banach space $E$, the space of all Lipschitz compact operators from $X$ to $E^*$ is isometrically isomorphic to $X^\#\widehat{\boxast}_{\pi'}E^*$.
\end{corollary}

By Theorem \ref{theo-dual-classic} and Corollary \ref{cor-nuevo2}, we have the following.

\begin{corollary}\label{cor-nuevo3}
Let $X$ be a pointed metric space such that $X^\#$ has the approximation property and let $E$ be a Banach space. Then $X^\#\widehat{\boxast}_{\pi'}E^*$ is isometrically isomorphic to $(X\widehat{\boxtimes}_\pi E)^*$ if and only if $\Lipo(X,E^*)$ is isometrically isomorphic to the space of Lipschitz compact operators from $X$ to $E^*$. 
\end{corollary}

We close this section with a new formula for the norm $\pi'$. By \cite[Lemma 1.1]{jsv}, the closed unit ball of the Lipschitz-free Banach space $\F(X)$ over a pointed metric space $X$ coincides with the closure of the convex balanced hull of the set $\{(\delta_x-\delta_y)/d(x,y)\colon x,y\in X,\; x\neq y\}$ in $(X^\#)^*$, where $\delta_x$ is the evaluation functional at $x$ defined on $X^\#$. It is well known that every element in the convex balanced hull of that set is of the form
$$
\sum_{i=1}^n\lambda_i\frac{\delta_{x_i}-\delta_{y_i}}{d(x_i,y_i)}
$$
for some $n\in\N$, $\lambda_1,\ldots,\lambda_n\in\K$, $\sum_{i=1}^n|\lambda_i|\leq 1$ and $(x_1,y_1),\ldots,(x_n,y_n)\in X^2$ with $x_i\neq y_i$ for $i\in\{1,\ldots,n\}$.

\begin{definition}\label{def-inj-norm-dual}
For each $\sum_{j=1}^m  g_j\boxtimes \phi_j\in X^\#\boxast E^*$, define:
$$
\varepsilon\left(\sum_{j=1}^m  g_j\boxtimes \phi_j\right)=\sup\left\{\left|\sum_{j=1}^m \gamma(g_j)\left\langle\phi_j,e\right\rangle\right|\colon \gamma\in B_{\F(X)},\; e\in B_E\right\}.
$$
\end{definition}

Note that this supremum exists since, for all $\gamma\in B_{\F(X)}$ and $e\in B_E$, 
$$
\left|\sum_{j=1}^m \gamma(g_j)\left\langle\phi_j,e\right\rangle\right|
\leq\sum_{j=1}^m\left|\gamma(g_j)\left\langle\phi_j,e\right\rangle\right|
\leq\sum_{j=1}^m \left\|\gamma\right\|\Lip(g_j)\left\|\phi_j\right\|\left\|e\right\|
\leq\sum_{j=1}^m \Lip(g_j)\left\|\phi_j\right\|.
$$

\begin{theorem}
The associated Lipschitz norm $\pi'$ of $\pi$ on $X\boxtimes E$ is $\varepsilon$ on $X^\#\boxast E^*$. 
\end{theorem}

\begin{proof}
Let $\sum_{j=1}^m  g_j\boxtimes \phi_j\in X^\#\boxast E^*$. We have    
\begin{align*}
\pi'\left(\sum_{j=1}^m  g_j\boxtimes \phi_j\right)
&=\sup\left\{\left|\left(\sum_{j=1}^m g_j\boxtimes\phi_j\right)\left(\sum_{i=1}^n \delta_{(x_i,y_i)}\boxtimes e_i\right)\right|\colon \pi\left(\sum_{i=1}^n \delta_{(x_i,y_i)}\boxtimes e_i\right)\leq 1\right\}\\
&=\sup\left\{\left|\left(\sum_{j=1}^m g_j\boxtimes\phi_j\right)\left(\delta_{(x,y)}\boxtimes e\right)\right|\colon \pi\left(\delta_{(x,y)}\boxtimes e\right)\leq 1\right\}\\
&=\varepsilon\left(\sum_{j=1}^m  g_j\boxtimes \phi_j\right).
\end{align*}
In order to justify these equalities, denote by $\alpha(\sum_{j=1}^m  g_j\boxtimes \phi_j)$ and $\beta(\sum_{j=1}^m  g_j\boxtimes \phi_j)$ the first and the second supremum which appear above. The first equality follows from Corollary \ref{pro-dual-norm}. To see that $\alpha(\sum_{j=1}^m  g_j\boxtimes \phi_j)\leq \beta(\sum_{j=1}^m  g_j\boxtimes \phi_j)$, we will use the easy fact that if $n\in\N$, $a_1,\ldots,a_n\in\R_0^+$ and $b_1,\ldots,b_n\in\R^+$, then 
$$
\frac{a_1+\cdots+a_n}{b_1+\cdots+b_n} \leq\max\left\{\frac{a_1}{b_1},\cdots,\frac{a_n}{b_n}\right\}.
$$
Fix $\sum_{i=1}^n \delta_{(x_i,y_i)}\boxtimes e_i\in X\boxtimes E$, nonzero. If $\sum_{i=1}^p \delta_{(x'_i,y'_i)}\boxtimes e'_i=\sum_{i=1}^n \delta_{(x_i,y_i)}\boxtimes e_i$, we have
\begin{align*}
\frac{\left|\left(\sum_{j=1}^m g_j\boxtimes\phi_j\right)\left(\sum_{i=1}^n \delta_{(x_i,y_i)}\boxtimes e_i\right)\right|}{\sum_{i=1}^p d(x'_i,y'_i)\left\|e'_i\right\|}
&=\frac{\left|\left(\sum_{j=1}^m g_j\boxtimes\phi_j\right)\left(\sum_{i=1}^p \delta_{(x'_i,y'_i)}\boxtimes e'_i\right)\right|}{\sum_{i=1}^p d(x'_i,y'_i)\left\|e'_i\right\|}\\
&\leq\frac{\sum_{i=1}^p\left|\left(\sum_{j=1}^m g_j\boxtimes\phi_j\right)( \delta_{(x'_i,y'_i)}\boxtimes e'_i)\right|}{\sum_{i=1}^p d(x'_i,y'_i)\left\|e'_i\right\|}\\
&\leq\max\left\{\frac{\left|\left(\sum_{j=1}^m g_j\boxtimes\phi_j\right)( \delta_{(x'_i,y'_i)}\boxtimes e'_i)\right|}{d(x'_i,y'_i)\left\|e'_i\right\|}\colon 1\leq i\leq p\right\}\\
&\leq\beta\left(\sum_{j=1}^m  g_j\boxtimes \phi_j\right).
\end{align*}
By the definition of $\pi$, it follows that 
$$
\left|\left(\sum_{j=1}^m g_j\boxtimes\phi_j\right)\left(\sum_{i=1}^n \delta_{(x_i,y_i)}\boxtimes e_i\right)\right|
\leq\beta\left(\sum_{j=1}^m  g_j\boxtimes \phi_j\right)\pi\left(\sum_{i=1}^n \delta_{(x_i,y_i)}\boxtimes e_i\right).
$$
This ensures that $\alpha(\sum_{j=1}^m  g_j\boxtimes \phi_j)\leq \beta(\sum_{j=1}^m  g_j\boxtimes \phi_j)$. The converse inequality is clearly certain. 

Now, given any $\delta_{(x,y)}\boxtimes e\in X\boxtimes E$ with $0<\pi(\delta_{(x,y)}\boxtimes e)\leq 1$, we obtain 
$$
\left|\left(\sum_{j=1}^m g_j\boxtimes\phi_j\right)\left(\delta_{(x,y)}\boxtimes e\right)\right|
=\left|\sum_{j=1}^m (g_j(x)-g_j(y))\left\langle\phi_j,e\right\rangle\right|
=\left|\sum_{j=1}^m\left(\frac{\delta_x-\delta_y}{d(x,y)}\right)(g_j)\left\langle\phi_j,d(x,y)e\right\rangle\right|
\leq \varepsilon\left(\sum_{j=1}^m  g_j\boxtimes \phi_j\right)
$$
since $(\delta_x-\delta_y)/d(x,y)\in S_{\F(X)}$ and $d(x,y)e\in B_E$. Passing to the supremum we arrive at $\beta(\sum_{j=1}^m  g_j\boxtimes \phi_j)\leq\varepsilon(\sum_{j=1}^m  g_j\boxtimes \phi_j)$. 

Finally, we show that $\varepsilon(\sum_{j=1}^m  g_j\boxtimes \phi_j)\leq\pi'(\sum_{j=1}^m  g_j\boxtimes \phi_j)$. For any $n\in\N$, $\lambda_1,\ldots,\lambda_n\in\K$, $\sum_{i=1}^n|\lambda_i|\leq 1$, $(x_i,y_i)\in X^2$ and $x_i\neq y_i$ for each $i\in\{1,\ldots,n\}$, and $e\in B_E$, we have
\begin{align*}
\left|\sum_{j=1}^m\left(\sum_{i=1}^n \lambda_i\frac{\delta_{x_i}-\delta_{y_i}}{d(x_i,y_i)}\right)(g_j)\left\langle \phi_j,e\right\rangle\right|
&=\left|\sum_{j=1}^m\left(\sum_{i=1}^n \lambda_i\frac{g_j(x_i)-g_j(y_i)}{d(x_i,y_i)}\left\langle \phi_j,e\right\rangle\right)\right|\\
&=\left|\sum_{j=1}^m\left(\sum_{i=1}^n \frac{\lambda_i}{d(x_i,y_i)}(g_j\boxtimes\phi_j)(\delta_{(x_i,y_i)}\boxtimes e)\right)\right|\\
&=\left|\left(\sum_{j=1}^m g_j\boxtimes\phi_j\right)\left(\sum_{i=1}^n\delta_{(x_i,y_i)}\boxtimes \frac{\lambda_i e}{d(x_i,y_i)}\right)\right|\\
&\leq\pi'\left(\sum_{j=1}^m  g_j\boxtimes \phi_j\right)
\end{align*}
since  
$$
\pi\left(\sum_{i=1}^n\delta_{(x_i,y_i)}\boxtimes \frac{\lambda_i e}{d(x_i,y_i)}\right)
\leq \sum_{i=1}^n\pi\left(\delta_{(x_i,y_i)}\boxtimes \frac{\lambda_i e}{d(x_i,y_i)}\right)
=\sum_{i=1}^n d(x_i,y_i)\frac{\left|\lambda_i\right|\left\|e\right\|}{d(x_i,y_i)}
=\left\|e\right\|\sum_{i=1}^n \left|\lambda_i\right|
\leq 1.
$$
By the density of the elements $\sum_{i=1}^n \lambda_i(\delta_{x_i}-\delta_{y_i})/d(x_i,y_i)$ in $B_{\F(X)}$, we infer that 
$$
\left|\sum_{j=1}^m\gamma(g_j)\left\langle\phi_j,e\right\rangle\right|
\leq\pi'\left(\sum_{j=1}^m g_j\boxtimes\phi_j\right)
$$
for all $\gamma\in B_{\F(X)}$ and $e\in B_E$. Taking supremum over all such $\gamma$ and $e$, we conclude that 
$\varepsilon(\sum_{j=1}^m  g_j\boxtimes \phi_j)\leq\pi'(\sum_{j=1}^m  g_j\boxtimes \phi_j)$ and this completes the proof.
\end{proof}

\bibliographystyle{amsplain}

\begin{thebibliography}{99}
\bibitem{ae} R. F. Arens and J. Eells Jr., \emph{On embedding uniform and topological spaces}, Pacific J. Math. \textbf{6} (1956), 397--403.
\bibitem{bjlps} S. Bates, W. B. Johnson, J. Lindenstrauss, D. Preiss and G. Schechtman, \emph{Affine approximation of Lipschitz functions and nonlinear quotients}, Geom. Funct. Anal. \textbf{9} (1999), no. 6, 1092--1127.
\bibitem{Castillo-hitchhiker} J. M. F. Castillo, \emph{The hitchhiker guide to categorical Banach space theory. Part {I}}, Extracta Math. \textbf{25} (2010), no. 2, 103--149.
\bibitem{ch11} J. A. Ch\'{a}vez-Dom\'{i}nguez, \emph{Duality for Lipschitz $p$-summing operators}, J. Funct. Anal. \textbf{261} (2011), no. 2, 387--407.
\bibitem{ch12} J. A. Ch\'{a}vez-Dom\'{i}nguez, \emph{Lipschitz $(q,p)$-mixing operators}, Proc. Amer. Math. Soc. \textbf{140} (2012), 3101--3115.
\bibitem{cz11} D. Chen and B. Zheng, \emph{Remarks on Lipschitz $p$-summing operators}, Proc. Amer. Math. Soc. \textbf{139} (2011), no. 8, 2891--2898.
\bibitem{cz12} D. Chen and B. Zheng, \emph{Lipschitz $p$-integral operators and Lipschitz $p$-nuclear operators}, Nonlinear Analysis \textbf{75} (2012), 5270--5282.
\bibitem{c} S. Chevet, \emph{Sur certains produits tensoriels topologiques d'espaces de Banach}. (French) Z. Wahrscheinlichkeitstheorie und Verw. Gebiete \textbf{11} (1969), 120--138.
\bibitem{fhhmz} Mari\'an Fabian, Petr Habala, Petr H\'ajek, Vicente Montesinos and V\'aclav Zizler, \emph{Banach space theory. The basis for linear and nonlinear analysis}. CMS Books in Mathematics/Ouvrages de Mathematiques de la SMC, Springer, New York, 2011.
\bibitem{fj} Jeffrey D. Farmer and William B. Johnson, \emph{Lipschitz $p$-summing operators}, Proc. Amer. Math. Soc. \textbf{137} (9) (2009), 2989--2995. 
\bibitem{Fuks} D. B. Fuks, \emph{Eckmann-Hilton duality and theory of functors in the category of topological spaces}, Russian Math. Surveys \textbf{21} (1966), no. 2 (128), 1--33.
\bibitem{Godard} A. Godard, \emph{Tree metrics and their Lipschitz-free spaces}, Proc. Amer. Math. Soc. \textbf{138} (2010), no. 12, 4311--4320.
\bibitem{gk} G. Godefroy and N. J. Kalton, \emph{Lipschitz-free Banach spaces}, Studia Math. \textbf{159} (2003), 121--141.
\bibitem{glz} G. Godefroy, G. Lancien, and V. Zizler, \emph{The non-linear geometry of Banach spaces after Nigel Kalton}, Rocky Mountain J. Math. (To appear).
\bibitem{g} A. Grothendieck, \emph{Produits Tensoriels Topologiques et Espaces Nucl\'{e}aires}, Memoirs American Mathematical Society 16, Providence, Rhode Island 1955.
\bibitem{jsv} A. Jim\'{e}nez-Vargas, J. M. Sepulcre and Mois\'{e}s Villegas-Vallecillos, \emph{Lipschitz compact operators}, J. Math. Anal. Appl., \textbf{415}, no. 2, (2014) 889--901.
\bibitem{j70}  J. A. Johnson, \emph{Banach spaces of Lipschitz functions and vector-valued Lipschitz functions}, Trans. Amer. Math. Soc. \textbf{148} (1970), 147--169.
\bibitem{j71} J. A. Johnson, \emph{Extreme points in tensor products and a theorem of de Leeuw}, Studia Math. \textbf{37} (1970/71), 159--162. 
\bibitem{k04} N. J. Kalton, \emph{Spaces of Lipschitz and H\"{o}lder functions and their applications}, Collect. Math. \textbf{55} (2004), 171--217.
\bibitem{Ryan} R. A. Ryan, \emph{Introduction to tensor products of Banach spaces}, Springer Monographs in Mathematics, Springer-Verlag London Ltd., London, 2002.
\bibitem{s} P. Saphar, \emph{Produits tensoriels d'espaces de Banach et classes d'applications linéaires}. (French) Studia Math. \textbf{38} (1970), 71--100. (errata insert).
\bibitem{schatten} R. Schatten, \emph{A theory of cross-spaces}, Annals of Mathematics Studies, no. 26. Princeton University Press, Princeton, N. J., 1950.
\bibitem{w}  N. Weaver, \emph{Lipschitz Algebras}, World Scientific Publishing Co., Singapore, 1999.

\end{thebibliography}

\end{document}